\documentclass[oneside,a4paper,reqno]{amsart}
\usepackage{pdfsync, enumitem}
\usepackage{amsmath,bm}
\usepackage{mathtools}
\usepackage{graphicx}
\usepackage[all,cmtip]{xy}
\usepackage{tikz-cd}
\usepackage{mathrsfs}
\usepackage{hyperref}
\usepackage{tikz-cd}
\usepackage{amssymb}
\usepackage{stackrel}
\usepackage{adjustbox}
\usepackage{multicol}
\usepackage{amsmath, amsthm, amscd, amssymb, latexsym, eucal}
\usepackage[all]{xy}
\def\serieslogo@{} \def\@setcopyright{} \makeatother
\usepackage{fullpage}
\usepackage[margin=2.5cm, top=2.5cm, bottom=2.5cm, footskip=1cm]{geometry}

\usepackage[colorinlistoftodos]{todonotes}

\usepackage{hyperref}
\usepackage{color}
\usepackage{cite}
\usepackage{quiver}

\makeatletter
\renewcommand*\env@matrix[1][c]{\hskip -\arraycolsep
	\let\@ifnextchar\new@ifnextchar
	\array{*\c@MaxMatrixCols #1}}
\makeatother

\usepackage{color}

\numberwithin{equation}{section}
\newtheorem{thm}{Theorem}[section]
\newtheorem*{main-thm}{Theorem}
\newtheorem*{Auslander-thm}{Auslander's Theorem}

\newtheorem{cor}[thm]{Corollary}
\newtheorem{lem}[thm]{Lemma}
\newtheorem{prop}[thm]{Proposition}

\newtheorem*{thmA}{Theorem~A}
\newtheorem*{thmB}{Theorem~B}

\theoremstyle{definition}
\newtheorem{defn}[thm]{Definition}
\newtheorem{rem}[thm]{Remark}
\newtheorem{exmp}[thm]{Example}

\newtheorem{con}{Convention}

\newtheorem*{connot}{Notation and Conventions}

\newcommand\padova{%
\mathrel{\ooalign{\hss{\scalebox{0.5}{$\longleftrightarrow$}}\hss\cr%
\kern0.9ex\raise0.55ex\hbox{\scalebox{0.5}{$\boldsymbol{\bm{\vert}}$}}}}}
\newcommand\hfpadova{%
\mathrel{\ooalign{\hss{\scalebox{0.5}{$\longleftrightarrow$}}\hss\cr%
\kern0.9ex\raise0.55ex\hbox{\scalebox{0.5}{$\boldsymbol{\mathsf{h}\mathsf{f}\ \ }$}}}}}
\newcommand\rpadova{%
\mathrel{\ooalign{\hss{\scalebox{0.5}{$\longrightarrow$}}\hss\cr%
\kern0.9ex\raise0.55ex\hbox{\scalebox{0.5}{$\boldsymbol{\bm{\vert}}$}}}}}
\newcommand\lpadova{%
\mathrel{\ooalign{\hss{\scalebox{0.5}{$\longleftarrow$}}\hss\cr%
\kern0.9ex\raise0.55ex\hbox{\scalebox{0.5}{$\boldsymbol{\bm{\vert}}$}}}}}

\newcommand{\T}{\mathcal T}

\DeclareMathOperator*{\gd}{\mathsf{gl.dim}}

\DeclareMathOperator*{\Mod}{\mathsf{Mod}-\!}
\DeclareMathOperator*{\dgMod}{\mathsf{dgMod}-\!}

\DeclareMathOperator*{\smod}{\mathsf{mod}-\!}

\DeclareMathOperator*{\lMod}{\!-\mathsf{Mod}}

\DeclareMathOperator*{\proj}{\mathsf{proj}-\!}

\DeclareMathOperator*{\Inj}{\mathsf{Inj}-\!}

\DeclareMathOperator*{\Proj}{\mathsf{Proj}-\!}

\DeclareMathOperator*{\Add}{\mathsf{Add}}

\usepackage{stackengine}

\newsavebox{\proofbox}
\savebox{\proofbox}{\begin{picture}(7,7)%
	\put(0,0){\framebox(7,7){}}\end{picture}}

\usepackage{latexsym}
\usepackage{pstricks}
\usepackage{comment}

\usepackage[greek,english]{babel}

\begin{document}

\title{Homological Aspects of Separable Extensions of Triangulated Categories}

\author[karakikes]{Miltiadis Karakikes}
\address{Department of Mathematics, National and Kapodistrian University of Athens Panepistimioupolis, 15784 Athens, Greece}
\email{miltoskar@math.uoa.gr}

\author[Kostas]{Panagiotis Kostas}
\address{Department of Mathematics, Aristotle University of Thessaloniki, Thessaloniki 54124, Greece}
\email{pkostasg@math.auth.gr}

\subjclass[2020]{18G80, 16E35, 18G20, 16E10, 16E65}
\keywords{compactly generated triangulated category, global dimension, Gorenstein algebras, separable monad, Eilenberg-Moore category, differential graded algebras, equivariant category, singularity category}
\thanks{\textbf{Acknowledgements:} M.K. was supported by the Hellenic Foundation for Research and Innovation (HFRI) under the 4th Call for HFRI Ph.D. Fellowships (Fellowship Number: 9270). We thank C.~Psaroudakis for comments on an earlier draft and F.~Andriopoulos for helpful discussions.}

\begin{abstract}
   We investigate the homological behaviour of compactly generated triangulated categories under separable extensions. We show that homological invariants (finiteness of global dimension, gorensteinness and regularity) are preserved under such extensions. We also establish a relation between singularity categories in this setting, proving that the singularity category of a separable extension is equivalent, up to retracts, to a separable extension of the singularity category. Our results unify and extend classical phenomena from commutative and equivariant algebra, and provide new examples involving separable extensions of rings, quotient schemes, and skew group dg algebras. 
\end{abstract}

\date{\today}
\maketitle
\setcounter{tocdepth}{1}

\section{Introduction}

Separability is a recurrent theme in algebra. In commutative algebra, it is closely linked to étale extensions, while in the noncommutative setting it naturally appears in skew group rings associated with finite group actions of invertible order. Such constructions typically preserve homological invariants, including finiteness of global dimension, gorensteinness, and regularity. This behaviour was first established in \cite{EGA} for \'etale morphisms of schemes 
and in \cite{reiten_riedtmann} for skew group algebras. 

Balmer \cite{balmer, balmer2} studied the category of modules over a separable exact monad on a triangulated category. He showed that the latter, under mild assumptions, admits a unique canonical triangulated structure. Both the derived category of an \'etale extension and the derived category of a skew group ring can, then, be described as categories of modules over the derived category of their respective base rings.

The purpose of this paper is to formulate and exploit the homological behaviour of separable extensions in a broader and more conceptual setting. Rather than working directly with rings and modules, we base ourselves on the work of Balmer, and use instead triangulated categories and separable monads. In doing so, we provide a common framework that not only explains, but also extends, the homological behaviour observed in classical commutative and equivariant settings. As for the homological part, we use the intrinsic notions of finite global dimension, gorensteinness and regularity of compactly generated triangulated categories, defined in \cite{regular}. 

Our first theorem showcases the interplay of homological invariants for separable constructions; see Theorem \ref{comparing regularity} and Proposition \ref{comparing regularity 2}. Below we present, for convenience, only its ``equivariant'' form, which is proved in Corollary \ref{regularity for equivariant}.

\begin{thmA}
    Let $\mathcal{T}$ be a compactly generated triangulated category and assume an action of a finite group $G$ on $\mathcal{T}$ such that $|G|$ is invertible in $\mathcal{T}$ and further that $\mathcal{T}^{G}$ is canonically triangulated. Consider the following conditions that may be satisfied.
    \begin{itemize} 
        \item[\textnormal{(i)}] $\mathcal{T}$ has finite global dimension.
        \item[\textnormal{(ii)}] $\mathcal{T}$ is Gorenstein. 
        \item[\textnormal{(iii)}] $\mathcal{T}$ is regular \textnormal{(}when $\mathcal{T}$ is $k$-linear\textnormal{)}.
    \end{itemize}
    Then $\mathcal{T}$ satisfies any of \textnormal{(i)-(iii)} if and only if $\mathcal{T}^G$ satisfies the same condition.
\end{thmA}

The above theorem, together with its more general form, subsume the homological behaviour of skew group algebras observed in \cite{reiten_riedtmann}, cover the transfer of regularity for proper schemes along a group action (Corollary \ref{regularity for schemes}) or an \'etale morphism (Corollary \ref{regularity for etale maps}) and also provide completely new examples of smooth and Gorenstein skew group dg algebras (see Corollary~\ref{homological invariants for skew group dg algebras}). The assumption that $\mathcal{T}^G$ is canonically triangulated is mild; as already mentioned, $\mathcal{T}^{G}$ admits a canonical pre-triangulated structure and in practice it is triangulated; see also Convention \ref{convention}. 

There is a notion of a singularity category for $k$-linear compactly generated triangulated categories, introduced in \cite[Definition 5.7]{regular}, which (under mild assumptions) recovers the singularity category of a ring or a scheme, see Example \ref{examples of singularity}. Our second main result is Theorem \ref{singularity}, which reveals a relation between the singularity category of modules over a separable monad and the modules over an induced monad on the singularity category. Below we present the equivariant case, which is proved in Corollary~\ref{equivariant singularity}. 

\begin{thmB}
    Assume the setup of Theorem A and that $\mathcal{T}$ is $k$-linear. Then, there is an exact functor 
    \[
    (\mathcal{T}^G)^\mathsf{sg}_k \to (\mathcal{T}^\mathsf{sg}_k)^G
    \]
    which is an equivalence up to retracts. 
\end{thmB}

The above together with its more general form are a common framework for equivalences of great interest and include the following examples about certain singularity categories, all of which are new:
\begin{itemize}
    \item[$\bullet$] (Example \ref{singularity of separable}) an equivalence up to retracts $\mathsf{D}_{\mathsf{sg}}(A)\rightarrow \mathsf{M}\lMod_{\mathsf{D}_{\mathsf{sg}}(k)}$ for $A$ a finite dimensional, flat and separable algebra over a commutative ring $k$;
    \item[$\bullet$] (Example~\ref{singularity of quotient scheme}) an equivalence up to retracts $\mathsf{D}_{\mathsf{sg}}(X/G)\rightarrow \mathsf{D}_{\mathsf{sg}}(X)^G$ for a proper scheme $X$ and certain group actions $G\curvearrowright X$;  
    \item[$\bullet$] (Example \ref{singularity of skew group dg}) an equivalence up to retracts $\mathsf{D}_{\mathsf{sg}}(A\#G)\rightarrow \mathsf{D}_{\mathsf{sg}}(A)^G$ for a proper dg algebra $A$ and certain group actions $G\curvearrowright A$; \item[$\bullet$] (Example \ref{singularity of skew group dg 2}) an equivalence up to retracts $\mathsf{D}_{\mathsf{agk}}(A\#G)\rightarrow \mathsf{D}_{\mathsf{agk}}(A)^G$ for a connective dg algebra $A$ and certain group actions $G\curvearrowright A$, 
\end{itemize}
where $\mathsf{D}_{\mathsf{agk}}(A)$ stands for the \emph{Amiot-Guo-Keller singularity category} \cite{amiot, guo}. 

Both the homological notions of Theorem A and the singularity category of Theorem B are defined via a list of intrinsic subcategories associated to a compactly generated triangulated category $\mathcal{T}$, defined in \cite{regular}. The latter subcategories are denoted by $\mathcal{T}^{\mathsf{b}}$, $\mathcal{T}^{\mathsf{b}}_p$, $\mathcal{T}^{\mathsf{b}}_i$ and $\mathcal{T}^{\mathsf{b}}_c$ (defined when $\mathcal{T}$ is $k$-linear) and we postpone their explicit description until Definition \ref{def:subcategories}. Our main tool in this paper, upon which the above theorems are based, is the behaviour of these subcategories under a separable extension; see Theorem \ref{main theorem}. For instance under the setup of Theorem A, we prove that there is an equality
\[
(\mathcal{T}^{\mathsf{b}}_{?})^G=(\mathcal{T}^G)^{\mathsf{b}}_?
\]
of subcategories of $\mathcal{T}^G$ for every $?\in\{\varnothing,i,p,c\}$ (where $\mathcal{T}$ is $k$-linear for $?=c$), see Corollary \ref{subcategories for equivariant}. 

We accompany the above results with several observations of independent interest. For instance, we show in Corollary \ref{perissos} that the skew group ring of a quasi-excellent commutative noetherian ring of finite Krull dimension satisfies a technical condition used in \cite{regular} to compute the intrinsic subcategories of homotopy categories of injectives. Moreover, we prove in Corollary \ref{global dimension for equivariant via t structure} an equality of relative global dimensions for equivariant categories; a result which enhances Theorem A(i). Lastly, we include an appendix where we observe the behaviour of Brown-Comenetz dual objects for a separable extension of a compactly generated triangulated category, based on Balmer's description of the compact objects.

\subsection*{\textbf{Structure of the paper.}} In Section \ref{preliminaries} we collect preliminaries that are used throughout the paper. Section \ref{main results} contains some of the main results of the paper; in particular, in Subsection \ref{Intrinsic subcategories for separable extensions} we study structural properties of the intrinsic subcategories for separable extensions of triangulated categories and using these, we compare in Subsection \ref{Comparing homological invariants for separable extensions} various homological invariants. In Section \ref{Working with the examples} we apply the results of the previous section to particular examples, namely
equivariant triangulated categories (Subsection~\ref{Equivariant triangulated categories}) and separable extensions of rings and \'etale morphisms of schemes (Subsection~\ref{Separable extensions}). Section~\ref{singularity categories} deals with singularity categories and in Section \ref{Global dimension via t-structures} we study the global dimension of a triangulated category relative to a t-structure and its behaviour under a separable extension. Appendix \ref{Brown-Comenetz duals} contains a result about the Brown-Comenetz dual objects in a separable extension.

\begin{connot}
For a ring $R$, we write $\Mod{R}$ for the category of all right $R$-modules and we consider its subcategories $\Proj{R}$ of projective modules, $\proj{R}$ of finitely generated projective modules, $\Inj R$ of injective modules, and $\smod{R}$ of finitely presented modules. Whenever we say that a triangulated category has coproducts we mean set-indexed coproducts. Unless otherwise stated, $k$ denotes a commutative noetherian ring. By an Artin algebra (respectively, a Noether algebra) we mean a finitely generated algebra over a commutative artinian (respectively, noetherian) ring. We say that a positive integer $n$ is invertible in $\mathsf{A}$ if for every morphism $f$ there exists a morphism $g$ such that $f= ng$. If $\mathsf{A}$ is a $k$-linear category, then $n$ is invertible in $\mathsf{A}$ if and only if it is invertible in $k$.  
\end{connot} 

\pagebreak

\section{Preliminaries} \label{preliminaries}

We collect some preliminaries, that are used throughout the paper, regarding separable extensions of triangulated categories \cite{balmer, balmer2} and intrinsic homological algebra for triangulated categories \cite{regular}.  

\subsection{Separable extensions} Let $\mathsf{A}$ be an additive category. A \emph{monad} on $\mathsf{A}$ is a triple $(\mathsf{M},\eta,\mu)$ consisting of an additive endofunctor $\mathsf{M}\colon\mathsf{A}\rightarrow\mathsf{A}$, a natural transformation $\eta\colon \mathsf{Id}_{\mathsf{A}}\rightarrow\mathsf{M}$ called the \emph{unit} and a natural transformation $\mu\colon\mathsf{M}^2\rightarrow\mathsf{M}$ called \emph{multiplication}, subject to the following relations: 
\[
\mu\circ \mathsf{M}\mu=\mu\circ \mu\mathsf{M} \ \ \ \text{  and  } \ \ \ \mu\circ \mathsf{M}\eta=\mu\circ \eta\mathsf{M}.
\]
Further, a monad as above is called \emph{separable} if $\mu$ admits a section, i.e.\ there exists a natural transformation $\sigma\colon\mathsf{M}\rightarrow \mathsf{M}^2$ such that $\mu\circ\sigma=\mathsf{Id}_{\mathsf{M}}$ and moreover 
\[
\mathsf{M}\mu\circ\sigma\mathsf{M}=\sigma\mu=\mu\mathsf{M}\circ\mathsf{M}\sigma.
\]
When $\mathsf{A}$ is a pre-triangulated category, we further require that $\mathsf{M},\eta,\mu,\sigma$ commute with the suspension, see \cite{balmer} for details. Henceforth, we will be simply denoting a monad by $\mathsf{M}$. We now recall the definition of the Eilenberg-Moore category \cite{eilenberg_moore} associated to a monad, following the notation from \cite{balmer}. 

\begin{defn} \label{eilenberg moore}
    Let $\mathsf{M}$ be a monad on an additive category $\mathsf{A}$.
    \begin{itemize}
        \item[$\bullet$] An  \emph{$\mathsf{M}$-module} is a pair $(x,u)$ where $x$ is an object in $\mathsf{A}$ and $u$ is a morphism $\mathsf{M}(x)\rightarrow x$ satisfying $u\circ \mathsf{M}u=u\circ \mu_x$ and $1_x=u\circ \eta_x$.
        \item[$\bullet$] A morphism $f\colon(x,u)\rightarrow(y,v)$ of $\mathsf{M}$-modules is a morphism $f\colon x\rightarrow y$ such that $v\circ \mathsf{M}(f)=f\circ u$.
        \end{itemize}
        The \emph{Eilenberg-Moore category} of $\mathsf{M}$, henceforth denoted by $\mathsf{M}\lMod_{\mathsf{A}}$, is the category with objects the $\mathsf{M}$-modules and morphisms the morphisms of $\mathsf{M}$-modules.
\end{defn}

\begin{lem}  \label{adjunction} \textnormal{(\!\!\!\cite{eilenberg_moore})}
    Let $\mathsf{M}$ be a monad on an additive category $\mathsf{A}$. There exist functors, the forgetful functor $\mathsf{U}_{\mathsf{M}}$ and the induction functor $\mathsf{T}_{\mathsf{M}}$, defined as follows.  
    \begin{itemize}
        \item[\textnormal{(i)}] $\mathsf{U}_{\mathsf{M}}\colon \mathsf{M}\lMod_{\mathsf{A}}\rightarrow\mathsf{A}$ is given by $\mathsf{U}_{\mathsf{M}}(x,u)=x$ on objects and by $\mathsf{U}_{\mathsf{M}}(f)=f$ on morphisms. 
        \item[\textnormal{(ii)}]  $\mathsf{T}_{\mathsf{M}}\colon \mathsf{A}\rightarrow\mathsf{M}\lMod_{\mathsf{A}}$ is given by $\mathsf{T}_{\mathsf{M}}(x)=(\mathsf{M}(x),\mu_x)$ on objects and $\mathsf{T}_{\mathsf{M}}(f)=\mathsf{M}(f)$ on morphisms.  
    \end{itemize}
    Observe that $\mathsf{M}=\mathsf{U}_{\mathsf{M}}\mathsf{T}_{\mathsf{M}}$ and, moreover, it holds that $(\mathsf{T}_{\mathsf{M}},\mathsf{U}_{\mathsf{M}})$ is an adjoint pair.
\end{lem}

Recall that a functor $\mathsf{e} \colon \mathsf{D} \to \mathsf{C}$ between additive categories is \textit{separable} if for all $x,y\in\mathsf{D}$, the maps $ \mathsf{Hom}_{\mathsf{D}}(x,y) \to \mathsf{Hom}_{\mathsf{C}}(\mathsf{e}(x) , \mathsf{e}(y))$ admit retractions, which are natural in $x,y$. In the setup of Lemma \ref{adjunction}, the functor $\mathsf{U}_{\mathsf{M}}$ is separable if and only if $\mathsf{M}$ is a separable monad; see for instance \cite[Lemma 3.2]{xwchen}.

It is useful to observe that for any adjoint pair $\mathsf{l}\colon\mathsf{C}\rightarrow \mathsf{D}$ and $\mathsf{e}\colon \mathsf{D}\rightarrow \mathsf{C}$, where $\mathsf{l}$ is the left adjoint, the endofunctor $\mathsf{M}\coloneqq\mathsf{el}$ of $\mathsf{C}$ defines a monad, see \cite{eilenberg_moore, maclane categories} for details. Then, there is a \emph{comparison functor} $\Phi\colon \mathsf{D}\rightarrow \mathsf{M}\lMod_{\mathsf{C}}$ making the following diagram commute 
\begin{equation} \label{comparison}
\begin{tikzcd}
	{\mathsf{D}} && {\mathsf{C}} \\
	\\
	&& {\mathsf{M}\lMod_{\mathsf{C}}}
	\arrow["{{\mathsf{e}}}"', shift right, from=1-1, to=1-3]
	\arrow["\Phi"', curve={height=12pt}, dashed, from=1-1, to=3-3]
	\arrow["{{\mathsf{l}}}"', shift right=2, from=1-3, to=1-1]
	\arrow["{{\mathsf{T}_{\mathsf{M}}}}"', shift right, from=1-3, to=3-3]
	\arrow["{{\mathsf{U}_{\mathsf{M}}}}"', shift right=2, from=3-3, to=1-3]
\end{tikzcd}
\end{equation}
i.e.\ $\Phi\circ \mathsf{l}=\mathsf{T}_{\mathsf{M}}$ and $\mathsf{e}=\mathsf{U}_{\mathsf{M}}\circ \Phi$, see \cite{maclane categories}. In this case the functor $\mathsf{e}$ is separable if and only if $\mathsf{M}= \mathsf{el}$ is separable and $\Phi$ is an equivalence up to retracts, see for instance \cite[Proposition 3.5]{xwchen}. By an equivalence up to retracts we mean a functor which is fully faithful and every object in the target is a retract of an object in the image (this is equivalent to the induced functor between the idempotent completions being an equivalence; see for instance \cite[Lemma 3.4]{xwchen}).  

\begin{thm} \label{balmer1} (Balmer) Let $\mathcal{T}$ be a pre-triangulated category and $\mathsf{M}$ a separable exact monad on $\mathcal{T}$. If  $\mathsf{M}\lMod_{\mathcal{T}}$ is idempotent complete, then it admits a unique pre-triangulated structure making the functors $\mathsf{U}_{\mathsf{M}}$ and $\mathsf{T}_{\mathsf{M}}$ exact. 
\end{thm}
\begin{proof}
    We consider a triangle in $\mathsf{M}\lMod_{\mathcal{T}}$ to be distinguished if its image under $\mathsf{U}_{\mathsf{M}}$ is distinguished in $\mathcal{T}$. This gives the pre-triangulated structure on $\mathsf{M}\lMod_{\mathcal{T}}$ and moreover makes $\mathsf{T}_{\mathsf{M}}$ exact; see \cite[4.1 Theorem]{balmer} for details. 
\end{proof}

We observe that we can partially drop the idempotent completeness assumption.

\begin{lem} \label{pretriangulated without idempotent}
    Let $\mathcal{T}$ be a pre-triangulated category and $\mathsf{M}$ a separable exact monad on $\mathcal{T}$. Then $\mathsf{M} \lMod_{\mathcal{T}}$ admits a unique pre-triangulated structure making the functors $\mathsf{U}_\mathsf{M}$ and $\mathsf{T}_\mathsf{M}$ exact.  
\end{lem}
\begin{proof}
    A version of this was shown in \cite[Proposition 3.3]{sun} for equivariant triangulated categories and we proceed here with similar arguments. First off, it is known from \cite{balmer_schlichting} that $\mathcal{T}^{\natural}$, the idempotent completion of $\mathcal{T}$, admits a unique pre-triangulated structure making the inclusion $\mathcal{T}\hookrightarrow \mathcal{T}^{\natural}$ exact. Moreover, it follows from \cite[Lemma 3.11]{sun} that we can lift the adjunction between $\mathsf{M}\lMod_{\mathcal{T}}$ and $\mathcal{T}$ to an adjunction between their respective idempotent completions as in the following diagram. 
\[\begin{tikzcd}
	{\mathsf{M}\lMod_{\mathcal{T}}} && {\mathcal{T}} \\
	\\
	{(\mathsf{M}\lMod_{\mathcal{T}})^{\natural}} && {\mathcal{T}^{\natural}} \\
	\\
	&& {\mathsf{M}^{\natural}\lMod_{\mathcal{T}^{\natural}}}
	\arrow["{\mathsf{U}_{\mathsf{M}}}"', shift right=2, from=1-1, to=1-3]
	\arrow[hook, from=1-1, to=3-1]
	\arrow["{\mathsf{T}_{\mathsf{M}}}"', shift right=2, from=1-3, to=1-1]
	\arrow[hook, from=1-3, to=3-3]
	\arrow["{\mathsf{U}_{\mathsf{M}}^{\natural}}"', shift right=2, from=3-1, to=3-3]
	\arrow["\Phi"', curve={height=12pt}, from=3-1, to=5-3]
	\arrow["{\mathsf{T}_{\mathsf{M}}^{\natural}}"', shift right=2, from=3-3, to=3-1]
	\arrow["{\mathsf{T}_{\mathsf{M}^{\natural}}}"', shift right=2, from=3-3, to=5-3]
	\arrow["{\mathsf{U}_{\mathsf{M}^{\natural}}}"', shift right=2, from=5-3, to=3-3]
\end{tikzcd}\]
Regarding the lower part of the diagram, the functor $\mathsf{M}^{\natural}$ is defined to be the composite $\mathsf{U}_{\mathsf{M}}^{\natural}\circ \mathsf{T}_{\mathsf{M}}^{\natural}$, which is a monad on $\mathcal{T}^{\natural}$ and yields the comparison functor $\Phi$ (see (\ref{comparison})). Since $\mathsf{U}_{\mathsf{M}}$ is separable, it follows by \cite[Lemma 3.11]{sun} that $\mathsf{U}_{\mathsf{M}}^{\natural}$ is separable, and therefore we infer from \cite[Proposition 3.5]{xwchen} that $\Phi$ is an equivalence up to retracts, thus an equivalence as $(\mathsf{M}\lMod_{\mathcal{T}})^{\natural}$ is idempotent complete. Since $\mathcal{T}^{\natural}$ is idempotent complete, we infer that $\mathsf{M}^{\natural}\lMod_{\mathcal{T}^{\natural}}$ is idempotent complete (see \cite[2.6 Remark]{balmer}) and therefore we know from Theorem \ref{balmer1} that the latter admits a unique pre-triangulated structure such that $\mathsf{U}_{\mathsf{M}^{\natural}}$ and $\mathsf{T}_{\mathsf{M}^{\natural}}$ are exact. In view of the equivalence $\Phi$, there is a unique pre-triangulated structure on $(\mathsf{M}\lMod_{\mathcal{T}})^{\natural}$ making $\mathsf{U}_{\mathsf{M}}^{\natural}$ and $\mathsf{T}_{\mathsf{M}}^{\natural}$ exact. Since $\mathsf{M}\lMod_{\mathcal{T}}$, as a full subcategory of its idempotent completion, is closed under shifts and cones, there is an induced pre-triangulated structure, which makes $\mathsf{U}_{\mathsf{M}}$ and $\mathsf{T}_{\mathsf{M}}$ exact - by using the diagram above; explicitly, a triangle in $\mathsf{M}\lMod_{\mathcal{T}}$ is distinguished if and only if its image under $\mathsf{U}_{\mathsf{M}}$ is distinguished in $\mathcal{T}$ (which follows by the analogous description on $\mathsf{M}^{\natural}\lMod_{\mathcal{T}^{\natural}}$ - see Theorem~\ref{balmer1}) which also proves the uniqueness part (since any other pre-triangulated structure such that $\mathsf{U_M}$ is exact necessarily contains the one defined).  
\end{proof}

We will later refer to the above pre-triangulated structure as the ``canonical'' one and observe that under the setup of (\ref{comparison}), if $\mathsf{C}$ and $\mathsf{D}$ are pre-triangulated, then the comparison functor is also pre-triangulated with $\mathsf{M}\lMod_{\mathsf{C}}$ given the pre-triangulated structure of the results above. We will later need the following lemma, see the proof of \cite[2.10 Proposition]{balmer}.

\begin{lem} \label{summand}
Consider an adjunction $(\mathsf{l,e})$ of pre-triangulated categories where $\mathsf{e}$ is faithful. Then every object $x$ in the domain of $\mathsf{e}$ is a summand of $\mathsf{le}(x)$. In particular, if $\mathcal{T}$ is a pre-triangulated category and $\mathsf{M}$ is a separable exact monad on $\mathcal{T}$, then every object $x$ of $ \mathsf{M}\lMod_{\mathcal{T}}$ is a summand of $\mathsf{T}_{\mathsf{M}}\mathsf{U}_{\mathsf{M}}(x)$.
\end{lem}

Although Theorem \ref{balmer1} produces, in principle, only a pre-triangulated category, all the examples we will work with (and all the pre-triangulated categories that we know of) satisfy the octahedral axiom. For this reason, we agree on the following convention. 
\begin{con} \label{convention}
    From now on we will only work with triangulated categories and assume that (under the assumptions of Theorem \ref{balmer1}) the category  $\mathsf{M}\lMod_{\mathcal{T}}$, with the canonical pre-triangulated structure of Theorem \ref{balmer1} and Lemma \ref{pretriangulated without idempotent}, is triangulated. Alternatively, the reader may assume that whenever we say ``triangulated'', we mean a triangulation of order $n\geq 3$ in the sense of \cite[Section 5]{balmer}, in which case it follows that the pre-triangulated structure given on $\mathsf{M}\lMod_{\mathcal{T}}$ is triangulated, see \cite[5.1.7 Main Theorem]{balmer}.
\end{con}

For a triangulated category $\mathcal{T}$ with coproducts, we say that a monad $\mathsf{M}$ is \emph{smashing} \cite{balmer2} if it commutes with arbitrary coproducts. We denote the set of compact object of $\mathcal{T}$ by $\mathcal{T}^{\mathsf{c}}$ and recall that $\mathcal{T}$ is \textit{compactly generated} if $\mathcal{T}^{\mathsf{c}}$ is small and $(\mathcal{T}^{\mathsf{c}})^\perp =0$.
The following result is proved in \cite[4.2 Theorem]{balmer2}.

\begin{thm} \label{balmer2} (Balmer) Let $\mathcal{T}$ be a compactly generated triangulated category and $\mathsf{M}$ a separable exact smashing monad on $\mathcal{T}$. Then the category $\mathcal{D}\coloneqq \mathsf{M}\lMod_{\mathcal{T}}$ is compactly generated and 
\[
\mathcal{D}^{\mathsf{c}}=\mathsf{thick}(\mathsf{T}_{\mathsf{M}}(\mathcal{T}^{\mathsf{c}})).
\]
If, moreover, $\mathsf{M}(\mathcal{T}^{\mathsf{c}})\subseteq \mathcal{T}^{\mathsf{c}}$, then the equality $(\mathsf{M}\lMod_\mathcal{T})^{\mathsf{c}}=\mathsf{M}\lMod_{\mathcal{T}^{\mathsf{c}}}$
of subcategories of $\mathcal{D}$ holds. 
\end{thm}

In order to explain the second claim of the theorem above, we consider any subcategory $\mathcal{U}$ of $\mathcal{T}$. Then, if the monad $\mathsf{M}$ restricts to $\mathcal{U}$, we view $\mathsf{M}\lMod_{\mathcal{U}}$ as a full subcategory of $\mathsf{M}\lMod_{\mathcal{T}}$ in the obvious way. This justifies the terminology above and we use the same throughout the paper. The latter full subcategory is precisely the subcategory of $\mathsf{M}\lMod_{\mathcal{T}}$ that consists of the objects $(x,u)$ with $x\in\mathcal{U}$. In fact, we have the following. 
\begin{lem} \label{underdog}
    Let $\mathcal{T}$ be a triangulated category, $\mathsf{M}$ a separable exact monad and consider $\mathcal{D}=\mathsf{M}\lMod_{\mathcal{T}}$. Assume that the adjunction $(\mathsf{T}_{\mathsf{M}},\mathsf{U}_{\mathsf{M}})$ restricts to an adjunction between full triangulated subcategories $\mathcal{V}$ and $\mathcal{U}$ of $\mathcal{D}$ and $\mathcal{T}$ respectively. If $\mathcal{V}$ is thick, then $\mathcal{V}=\mathsf{M}\lMod_{\mathcal{U}}$ as subcategories of $\mathcal{D}$. 
\end{lem}
\begin{proof}
    It is enough to show that $\mathcal{V}$ consists precisely of the objects $(x,u)$ of $\mathcal{D}$ such that $x$ belongs in $\mathcal{U}$. One inclusion is clear. For the converse, we consider an object $(x,u)$ of $\mathcal{D}$ such that $x$ belongs in $\mathcal{U}$, for which it follows that $\mathsf{T}_{\mathsf{M}}(x)$ belongs in $\mathcal{V}$. Since $(x,u)$ is a summand of the latter and $\mathcal{V}$ is assumed to be thick, the claim follows. 
\end{proof}

\subsubsection{Separable algebras} Let $R$ be a commutative ring. An $R$-algebra $A$ is called \emph{separable} if the multiplication map $A\otimes_RA\to A$ admits a section. The following result was shown in \cite[6.5~Theorem]{balmer} and applies in particular to \'etale algebras, see \cite[Theorem]{balmer}.

\begin{thm} \label{etale Balmer}
    Let $R$ be a commutative ring and $A$ a flat separable $R$-algebra. The functor $\mathsf{M}=-\otimes_RA$ is a monad on $\mathsf{D}(R)$ \textnormal{(}i.e.\ $A$ is a ``ring object'' in $\mathsf{D}(R)$\textnormal{)} and there is an equivalence 
    \[
    \mathsf{D}(A)\simeq \mathsf{M}\lMod{_{\mathsf{D}(R)}}
    \]
    of triangulated categories.     
\end{thm}

\subsubsection{Equivariant triangulated categories} \label{equivariant triangulated categories}
For an in depth exposition of group actions on categories we refer to \cite{xwchen,elagin,sun}; instead let us quickly recall some facts. 

Let $\mathsf{A}$ be an additive category and $G$ a finite group. Denote by $\mathsf{Aut}(\mathsf{A})$ the group of isomorphism classes of autoequivalences of $\mathsf{A}$. A group action of $G$ on $\mathsf{A}$ is a group homomorphism $\rho \colon G \to \mathsf{Aut}(\mathsf{A})$, $g \mapsto \rho_g$, such that the natural isomorphisms $\theta_{g,h}\colon \rho_g \rho_h \to \rho_{gh}$ satisfy $\theta_{g,hk} \circ \rho_g\theta_{h,k} = \theta_{gh,k} \circ \theta_{g,h}\rho_k$ (i.e.~they are associative). Given a group action on (any) category, one defines the \emph{equivariant category} $\mathsf{A}^G$ to be the category with objects pairs $(x, \phi)$, where $\phi=\{ \phi_g \} _{g \in G}$ is a family of isomorphisms $\phi_g \colon x \to \rho_g(x)$ indexed by $G$, called a \emph{linearization} of $x$, satisfying $\theta_{g,h} \circ \rho_g(\phi_h) \circ \phi_g =  \phi_{gh}$. The morphisms of the equivariant category $\mathsf{A}^G$ are the morphisms of $\mathsf{A}$ that respect the linearizations. Note that $\mathsf{A}^G$ is additive (resp. abelian) when $\mathsf{A}$ is additive (resp. abelian).

Under the above setup, there exist the following functors: the \emph{forgetful} functor $\mathsf{F}\colon\mathsf{A}^{G}\rightarrow\mathsf{A}$ and the \emph{induction} functor $\mathsf{Ind}\colon\mathsf{A}\rightarrow\mathsf{A}^{G}$, which form an adjoint triple $(\mathsf{Ind},\mathsf{F},\mathsf{Ind})$. If $|G|$ is invertible in $\mathsf{A}$, then the monad $\mathsf{M}=\mathsf{F}\circ \mathsf{Ind}$, which is given by $\mathsf{M}(x)=\oplus_{g\in G}\rho_g(x)$, is separable; see for instance \cite[Lemma~2.14(2)]{sun}. Moreover, we know that the comparison functor $\mathsf{A}^{G}\rightarrow\mathsf{M}\lMod_{\mathsf{A}}$ is an equivalence; see for instance \cite[Proposition 3.11]{elagin}.

For a triangulated category $\mathcal{T}$, an action of a finite group $G$ on $\mathcal{T}$ is an action as above, where $\mathsf{Aut}(\mathcal{T})$ denotes the group of isomorphism classes of exact autoequivalences of $\mathcal{T}$.
By the preceding discussion, Theorem \ref{balmer1} and Lemma \ref{pretriangulated without idempotent}, it follows that $\mathcal{T}^{G}$ admits a unique pre-triangulated structure such that the forgetful functor is exact; see \cite[Proposition 3.3]{sun} for details. In this case, the induced monad $\mathsf{M}$ is exact and smashing, as a consequence of the ambidextrous adjunction.

\begin{prop}\label{compactly generated equivariant}
    Assume an action of a finite group $G$  on a triangulated category $\mathcal{T}$ which has arbitrary coproducts, such that $|G|$ is invertible in $\mathcal{T}$. Then $\mathcal{T}$ is compactly generated if and only if $\mathcal{T}^{G}$ is compactly generated, in which case
    \[
    (\mathcal{T}^G)^{\mathsf{c}}=(\mathcal{T}^{\mathsf{c}})^G
    \]
    as subcategories of $\mathcal{T}^G$. 
\end{prop}
\begin{proof}
    One direction is implied by Theorem~\ref{balmer2}. For the converse it is direct to verify that a set of compact generators $\{(x_i,\phi_i)\}$ of $\mathcal{T}^G$ is mapped to a set of compact generators of $\mathcal{T}$ via the forgetful functor. Moreover, any object $(x,\phi)$ of $\mathcal{T}^{G}$ is a summand of $\mathsf{Ind}(x)$, by Lemma~\ref{summand}, from which we infer that $(x,\phi)$ is compact if and only if $x$ is compact, and thus the equality in the statement holds.
\end{proof}

We also need the following result, which is due to Elagin \cite[Theorem 7.1, Remark 7.2]{elagin}, phrased in a slightly more general form; see \cite[Example 3.20]{sun}. 

\begin{thm} \label{equivariant derived} 
Let $\mathcal{A}$ be an abelian category and $G$ a finite group acting on $\mathcal{A}$ such that $|G|$ is invertible in $\mathcal{A}$. Then $\mathsf{D^?}(\mathcal{A})$ admits a natural $G$-action and there exists an exact functor $\mathsf{D^?}(\mathcal{A}^G) \xrightarrow[]{} \mathsf{D^?}(\mathcal{A})^G$ which is equivalence up to retracts, where $\mathsf{?}$ is any of the symbols $\{\varnothing, \mathsf{+,-,b} \}$.
\end{thm}

\subsection{Intrinsic subcategories of a triangulated category and regularity} We recall several concepts from \cite{regular}. Given a triangulated category $\mathcal{T}$ (typically with coproducts) and a class $\mathcal{X}$ of objects in $\mathcal{T}$, we consider the following subcategories of $\mathcal{T}$.
\[
\begin{aligned}
\mathcal{X}^{\padova}&\coloneqq \{t\in\mathcal{T} \ | \  \forall x\in\mathcal{X}, \exists n_x\in\mathbb{Z}: \mathsf{Hom}_{\mathcal{T}}(x,t[n])=0 \text{ for } |n|>n_x\}.  \\ 
^{\padova}\mathcal{X}&\coloneqq \{t\in\mathcal{T} \ | \  \forall x\in\mathcal{X}, \exists n_x\in\mathbb{Z}: \mathsf{Hom}_{\mathcal{T}}(t,x[n])=0 \text{ for } |n|>n_x\}.
\end{aligned}
\]
When $\mathcal{T}$ is $k$-linear, for a commutative noetherian ring $k$, we also consider
\[
\begin{aligned}
    \mathcal{X}^{\hfpadova}\coloneqq \{t\in\mathcal{T}\ |\ \forall x\in\mathcal{X}, \underset{n\in\mathbb{Z}}{\oplus}\mathsf{Hom}_{\mathcal{T}}(x,t[n])\in\smod k\}.
\end{aligned}
\]
The above subcategories are thick and triangulated subcategories of $\mathcal{T}$.

\begin{defn} \label{def:subcategories}
    Let $\mathcal{T}$ be a triangulated category with coproducts. We define
    \begin{itemize}
        \item[(i)] the subcategory of \emph{bounded} objects $\mathcal{T}^{\mathsf{b}}\coloneqq (\mathcal{T}^{\mathsf{c}})^{\padova}$;
        \item[(ii)] the subcategory of \emph{bounded projective}
        objects $\mathcal{T}^{\mathsf{b}}_p\coloneqq {}^{\padova}(\mathcal{T}^{\mathsf{b}})$; 
        \item[(iii)] the subcategory of \emph{bounded injective}
        objects $\mathcal{T}^{\mathsf{b}}_i\coloneqq (\mathcal{T}^{\mathsf{b}})^{\padova}$.
    \end{itemize}
    When $\mathcal{T}$ is $k$-linear, we also consider
    \begin{itemize}
        \item[(iv)] the subcategory of \emph{bounded finite} objects $\mathcal{T}^{\mathsf{b}}_c\coloneqq (\mathcal{T}^{\mathsf{c}})^{\hfpadova}$. 
    \end{itemize}
\end{defn}

The motivation for the above subcategories arises from the derived category of a ring; see \cite[Proposition 3.3]{regular} for a proof and the references therein. 

\begin{prop} \label{derived categoery of a ring}
    Let $R$ be a ring and $\mathcal{T}=\mathsf{D}(R)$ its derived category. Then
    \[
    \mathcal{T}^{\mathsf{b}}=\mathsf{D}^{\mathsf{b}}(R), \ \ \mathcal{T}^{\mathsf{b}}_p=\mathsf{K}^{\mathsf{b}}(\Proj R), \ \ \mathcal{T}^{\mathsf{b}}_i=\mathsf{K}^{\mathsf{b}}(\Inj R)
    \]
    and if, additionally, $R$ is a Noether algebra, then $ \mathcal{T}^{\mathsf{b}}_c=\mathsf{D}^{\mathsf{b}}(\smod R)$.
\end{prop}

Recall that a ring $R$ is called \emph{Gorenstein} if projective modules have finite injective dimension and injective modules have finite projective dimension. Further, when $R$ is noetherian, we say that $R$ is \emph{regular} if finitely generated modules have finite projective dimension. Motivated by the derived category of a ring, the following definition was introduced in \cite{regular}.

\begin{defn} \label{notions of regularity}
    A compactly generated triangulated category $\mathcal{T}$:  
    \begin{itemize}
        \item[(i)] has \emph{finite global dimension} if $\mathcal{T}^{\mathsf{b}}=\mathcal{T}^{\mathsf{b}}_p$ (or equivalently $\mathcal{T}^{\mathsf{b}}=\mathcal{T}^{\mathsf{b}}_i$, see \cite[Lemma 4.1]{regular}); 
        \item[(ii)] is \emph{Gorenstein} if $\mathcal{T}^{\mathsf{b}}_p=\mathcal{T}^{\mathsf{b}}_i$;
    \end{itemize}
    and when $\mathcal{T}$ is $k$-linear, then it is 
    \begin{itemize}
        \item[(iii)] \emph{regular} if $\mathcal{T}^{\mathsf{b}}_c=\mathcal{T}^{\mathsf{c}}$. 
    \end{itemize}
\end{defn}

From the discussion above, it is clear that the defined homological concepts behave well with respect to the derived category of a ring. We refer to \cite{regular} for more examples; see in particular \cite[Theorem~4.3]{regular}.

\section{Main Results} \label{main results}

This section deals with the computation of the intrinsic subcategories, introduced previously, for a separable extension of a compactly generated triangulated category (see Theorem \ref{main theorem}), which is the main tool of the paper. In the second part of this section, we prove Theorem \ref{comparing regularity} and Proposition \ref{comparing regularity 2} which are predecessors of Theorem A in the introduction. 

\subsection{Intrinsic subcategories for separable extensions} \label{Intrinsic subcategories for separable extensions} For the proof of our main theorem, we need the following lemmata, which we recall for convenience.

\begin{lem} \textnormal{(\!\!\cite[Lemma 2.9, Lemma 5.2]{regular})} \label{perp of thick}
    Let $\mathcal{T}$ be a triangulated category with coproducts. For any collection of objects $\mathcal{X}$ of $\mathcal{T}$, we have 
    \[
    \mathcal{X}^{?}=\mathsf{thick}(\mathcal{X})^{?} \ \ \ \ \& \ \ \ \ {^?}\mathcal{X}={^?}\mathsf{thick}(\mathcal{X}),
    \]
    for any $?\in\{\padova,\hfpadova\}$.
\end{lem}

\begin{lem}\textnormal{(\!\!\cite[Lemma 2.11]{regular})} \label{perp under adjunction}
    Let $\mathsf{F}$ be an exact functor $\mathsf{F}\colon\mathcal{D}\rightarrow \mathcal{T}$ of triangulated categories and let $\mathcal{X}$ and $\mathcal{Y}$ be classes of objects in $\mathcal{D}$ and $\mathcal{T}$ respectively, such that $\mathsf{F}(\mathcal{X})\subseteq \mathcal{Y}$. Then, 
    \begin{itemize}
        \item[\textnormal{(i)}]if $\mathsf{F}$ admits a right adjoint $\mathsf{G}$, then $\mathsf{G}(\mathcal{Y}^{?})\subseteq \mathcal{X}^{?}$;
        \item[\textnormal{(ii)}] if $\mathsf{F}$ admits a left adjoint $\mathsf{G}$, then $\mathsf{G}(^{?}\mathcal{Y})\subseteq {^?}\mathcal{X}$,
    \end{itemize}
    for any $?\in\{\padova,\hfpadova\}$.
\end{lem}

\pagebreak

\begin{thm} \label{main theorem}
    Let $\mathcal{T}$ be a compactly generated triangulated category with coproducts and $\mathsf{M}$ a separable exact smashing monad on $\mathcal{T}$. The following hold for $\mathcal{D}\coloneqq \mathsf{M}\lMod_{\mathcal{T}}$. 
    \begin{itemize}
        \item[\textnormal{(i)}] An object $(x,u)$ of $\mathcal{D}$ belongs in $\mathcal{D}^{\mathsf{b}}$ if and only if $x\in\mathcal{T}^{\mathsf{b}}$. In particular, if $\mathsf{M}(\mathcal{T}^{\mathsf{b}})\subseteq \mathcal{T}^{\mathsf{b}}$, then $(\mathsf{M}\lMod_{\mathcal{T}})^{\mathsf{b}}=\mathsf{M}\lMod_{\mathcal{T}^{\mathsf{b}}}$ as subcategories of $\mathcal{D}$. 
        \item[\textnormal{(ii)}] Given an object $(x,u)$ of $\mathcal{D}$ such that $x$ belongs in $\mathcal{T}^{\mathsf{b}}_i$, then $(x,u)$ belongs in $\mathcal{D}^{\mathsf{b}}_i$. If moreover $\mathsf{M}(\mathcal{T}^{\mathsf{b}})\subseteq \mathcal{T}^{\mathsf{b}}$, then $(x,u)$ belongs in $\mathcal{D}^{\mathsf{b}}_i$ if and only if $x$ belongs in $\mathcal{T}^{\mathsf{b}}_i$. If additionally to all the above, the inclusion $\mathsf{M}(\mathcal{T}^{\mathsf{b}}_i)\subseteq \mathcal{T}^{\mathsf{b}}_i$ holds, then $(\mathsf{M}\lMod_{\mathcal{T}})^{\mathsf{b}}_i=\mathsf{M}\lMod_{\mathcal{T}^{\mathsf{b}}_i}$ as subcategories of $\mathcal{D}$.

        \item[\textnormal{(iii)}] Given an object $(x,u)$ of $\mathcal{D}$ such that $x$ belongs in $\mathcal{T}^{\mathsf{b}}_p$, then $(x,u)$ belongs in $\mathcal{D}^{\mathsf{b}}_p$. In particular, if $\mathsf{U}_{\mathsf{M}}(\mathcal{D}^{\mathsf{b}}_p)\subseteq \mathcal{T}^{\mathsf{b}}_p$, then $(x,u)$ belongs in $\mathcal{D}^{\mathsf{b}}_p$ if and only if $x$ belongs in $\mathcal{T}^{\mathsf{b}}_p$. Moreover, in this case, $(\mathsf{M}\lMod_{\mathcal{T}})^{\mathsf{b}}_p=\mathsf{M}\lMod_{\mathcal{T}^{\mathsf{b}}_p}$ as subcategories of $\mathcal{D}$. 
        \item[\textnormal{(iv)}] Assume that $\mathcal{T}$ is $k$-linear \textnormal{(}which also implies that $\mathcal{D}$ is $k$-linear\textnormal{)}. Then, an object $(x,u)$ of $\mathcal{D}$ belongs in $\mathcal{D}^{\mathsf{b}}_c$ if and only if $x$ belongs in $\mathcal{T}^{\mathsf{b}}_c$. In particular, if the inclusion $\mathsf{M}({\mathcal{T}^{\mathsf{b}}_c})\subseteq \mathcal{T}^{\mathsf{b}}_c$ holds, then $(\mathsf{M}\lMod_{\mathcal{T}})^{\mathsf{b}}_c=\mathsf{M}\lMod_{\mathcal{T}^{\mathsf{b}}_c}$ as subcategories of $\mathcal{D}$. 
    \end{itemize}
\end{thm}
\begin{proof}
    (i) We know from Theorem \ref{balmer2} that $\mathcal{D}^{\mathsf{c}}=\mathsf{thick}(\mathsf{T}_{\mathsf{M}}(\mathcal{T}^{\mathsf{c}}))$, and therefore, by Lemma \ref{perp of thick}(i), we can deduce the equality $\mathcal{D}^{\mathsf{b}}=(\mathsf{T}_{\mathsf{M}}(\mathcal{T}^{\mathsf{c}}))^{\padova}$. Therefore, using the adjunction $(\mathsf{T}_{\mathsf{M}},\mathsf{U}_{\mathsf{M}})$, we see that $\mathcal{D}^{\mathsf{b}}$ consists of the objects $(x,u)$ of $\mathcal{D}$ such that $x$ belongs in $(\mathcal{T}^{\mathsf{c}})^{\padova}=\mathcal{T}^{\mathsf{b}}$. The last claim follows from the above.

    (ii) A morphism $ (y,v)\rightarrow (x,u)[n]$ where $(y,v)$ belongs in $\mathcal{D}^{\mathsf{b}}$ is given, in particular, by a morphism $y \rightarrow x[n]$ in $\mathcal{T}$, where $y\in\mathcal{T}^{\mathsf{b}}$ from (i). Therefore if $x\in\mathcal{T}^{\mathsf{b}}_i$, then the morphism $y \to x[n]$ is trivial for large enough $|n|$, which implies that $(y,v) \to (x,u)[n]$ is also trivial, for the same values of $n$, since $\mathsf{U_M}$ is faithful, and therefore  $(x,u)$ belongs in $\mathcal{D}^{\mathsf{b}}_i$. When $\mathsf{M}(\mathcal{T}^{\mathsf{b}})\subseteq \mathcal{T}^{\mathsf{b}}$, then using (i), it follows that the functor $\mathsf{T}_{\mathsf{M}}$ maps $\mathcal{T}^{\mathsf{b}}$ to $\mathcal{D}^{\mathsf{b}}$, and therefore given an object $(x,u)$ in $\mathcal{D}^{\mathsf{b}}_i$, it follows by Lemma~\ref{perp under adjunction} that $(x,u)$ belongs in $(\mathsf{T}_{\mathsf{M}}(\mathcal{T}^{\mathsf{b}}))^{\padova}$. Using the adjunction $(\mathsf{T}_{\mathsf{M}},\mathsf{U}_{\mathsf{M}})$ we see that $x\in\mathcal{T}^{\mathsf{b}}_i$. The last claim is an immediate consequence of the above. 
    
    (iii) The first claim is similar to the analogous claim in (ii) and the second is evident. Regarding the last claim, we observe that since $\mathsf{T}_{\mathsf{M}}$ preserves compact objects, it follows by Lemma~\ref{perp under adjunction}(i) that $\mathsf{U}_{\mathsf{M}}$ preserves bounded objects, i.e.\ $\mathsf{U}_{\mathsf{M}}(\mathcal{D}^{\mathsf{b}})\subseteq \mathcal{T}^{\mathsf{b}}$. As a consequence, by Lemma \ref{perp under adjunction}(ii), it follows that $\mathsf{T}_{\mathsf{M}}(\mathcal{T}^{\mathsf{b}}_p)\subseteq \mathcal{D}^{\mathsf{b}}_p$ and therefore $\mathsf{M}=\mathsf{U}_{\mathsf{M}}\circ \mathsf{T}_{\mathsf{M}}$ satisfies $\mathsf{M}(\mathcal{T}^{\mathsf{b}}_p)\subseteq \mathcal{T}^{\mathsf{b}}_p$.
    
    (iv) is similar to (i), using Lemma \ref{perp under adjunction} for $?=\hfpadova$.
\end{proof}

\begin{rem} \label{comonad}
    Assume now the setup of Theorem \ref{balmer2}. Since $\mathsf{M}$ is assumed to be smashing, the functor $\mathsf{U}_{\mathsf{M}}$ preserves coproducts and therefore, it follows from \cite[Theorem 4.1 and Theorem 5.1]{neeman} that it admits a right adjoint functor, which we denote by $\mathsf{R}_{\mathsf{M}}$. The situation is summarized below, where $\mathcal{D}\coloneqq \mathsf{M}\lMod_{\mathcal{T}}$.
\[
\begin{tikzcd}
\mathcal{D} \arrow[rr, "\mathsf{U}_{\mathsf{M}}"]                     &  & \mathcal{T} \arrow[ll, "\mathsf{T}_{\mathsf{M}}"', bend right] \arrow[ll, "\mathsf{R}_{\mathsf{M}}", bend left] \\
                                                                      &  &                                                                                                                 \\
\mathcal{D}^{\mathsf{c}} \arrow[uu, "\mathrm{S}_{\mathcal{D}}", hook] &  & \mathcal{T}^{\mathsf{c}} \arrow[uu, "\mathrm{S}_{\mathcal{T}}"', hook] \arrow[ll, "\mathsf{T}_{\mathsf{M}}"]   
\end{tikzcd}
\]
The composite $\mathsf{U}_{\mathsf{M}}\circ \mathsf{R}_{\mathsf{M}}$ defines a \emph{comonad} \cite{maclane categories} on $\mathcal{T}$, which we denote by $\mathsf{N}$. We then observe that if $\mathsf{N}(\mathcal{T}^{\mathsf{b}})\subseteq \mathcal{T}^{\mathsf{b}}$, then $\mathsf{U}_{\mathsf{M}}(\mathcal{D}^{\mathsf{b}}_p)\subseteq \mathcal{D}^{\mathsf{b}}_p$. Indeed, we first see that $\mathsf{R}_{\mathsf{M}}(\mathcal{T}^{\mathsf{b}})\subseteq \mathcal{D}^{\mathsf{b}}$, since $\mathsf{R}_{\mathsf{M}}(x)$ belongs in $\mathcal{D}^{\mathsf{b}}$ if and only if $\mathsf{U}_{\mathsf{M}}\mathsf{R}_{\mathsf{M}}(x)$ belongs in $\mathcal{T}^{\mathsf{b}}$ by Theorem \ref{main theorem}(i). Therefore it follows from Lemma \ref{perp under adjunction}(ii) that $\mathsf{U}_{\mathsf{M}}$ maps $\mathcal{D}^{\mathsf{b}}_p$ to $\mathcal{T}^{\mathsf{b}}_p$. That is to say that, although the second statement of Theorem \ref{main theorem} is not dual to the third, there is a dual statement when one passes to the category of comodules over $\mathsf{N}$. 
\end{rem}

Recall from \cite{regular} that for a triangulated category $\mathcal{T}$ and for any collection of objects, say $\mathcal{X}$, we denote by $\mathsf{broad}_{\Sigma}(\mathcal{X})$ the smallest thick triangulated subcategory of $\mathcal{T}$ that contains $\mathcal{X}$ and is closed under self-coproducts. The latter exists by \cite[Lemma 2.6]{regular} and coincides with $\mathsf{thick}(\cup_{x\in\mathcal{X}}\Add(x))$. For a noetherian ring $R$, we consider the following condition that may be satisfied 
\begin{equation} \label{paraskhnio2}
    \mathsf{D}^{\mathsf{b}}(\Mod R)=\mathsf{broad}_{\Sigma}(\mathsf{D}^{\mathsf{b}}(\smod R)). 
\end{equation}
The motivation for the above is that it appeared as a necessary tool for the computation of the intrinsic subcategories (of Definition \ref{def:subcategories}) in the case of $\mathsf{K}_{\mathsf{ac}}(\Inj R)$ and $\mathsf{K}(\Inj R)$, see \cite[Subsection 3.4]{regular}. This is satisfied for big classes of rings; for instance it holds for all Artin algebras. We will later investigate how this condition behaves under separable extensions (see Corollary \ref{perissos}), a topic of independent interest, for which we need the following. 

\begin{lem} \label{paraskhnio}
    Let $\mathcal{T}$ be a triangulated category with coproducts, $\mathsf{M}$ a separable exact smashing monad on $\mathcal{T}$ and $\mathcal{X}$ a subcategory of $\mathcal{T}$ satisfying the inclusion $\mathsf{M}(\mathcal{X})\subseteq \mathcal{X}$. Then $\mathsf{broad}_{\Sigma}(\mathsf{M}\lMod_{\mathcal{X}})=\mathsf{M}\lMod_{\mathsf{broad}_{\Sigma}(\mathcal{X})}$ as subcategories of $\mathsf{M}\lMod_{\mathcal{T}}$. 
\end{lem}
\begin{proof}
    In order to prove this, since $\mathsf{broad}_{\Sigma}(\mathsf{M}\lMod_{\mathcal{X}})$ is a thick subcategory of $\mathsf{M}\lMod_{\mathcal{T}}$, it is enough by Lemma \ref{underdog} to show that the adjunction $(\mathsf{T}_{\mathsf{M}},\mathsf{U}_{\mathsf{M}})$ restricts as below. 
    \[\begin{tikzcd}
	{\mathsf{broad}_{\Sigma}(\mathsf{M}\lMod_{\mathcal{X}})} && {\mathsf{broad}_{\Sigma}(\mathcal{X})}
	\arrow["{\mathsf{U}_{\mathsf{M}}}"', shift right=2, from=1-1, to=1-3]
	\arrow["{\mathsf{T}_{\mathsf{M}}}"', shift right=2, from=1-3, to=1-1]
\end{tikzcd}\]
    It is clear that for any exact functor $\mathsf{F}\colon \mathcal{C}\rightarrow \mathcal{C}'$ between any triangulated categories and for any subcategory $\mathcal{X}$ of objects of $\mathcal{C}$, we have $\mathsf{F}(\mathsf{thick}(\mathcal{X}))\subseteq \mathsf{thick}(\mathsf{F}(\mathcal{X}))$. Using this, the description of the ``broad closure'' and the fact that 
    \[
    \mathsf{T}_{\mathsf{M}}(\cup_{x\in\mathcal{X}}\Add(x))\subseteq \cup_{y\in\mathsf{M}\lMod_{\mathcal{X}}}\Add(y) \ \ \& \ \ \mathsf{U}_{\mathsf{M}}(\cup_{y\in\mathsf{M}\lMod_{\mathcal{X}}}\Add(y))\subseteq \cup_{x\in\mathcal{X}}\Add(x)
    \]
    (for the second inclusion use the fact that $\mathsf{U}_{\mathsf{M}}$ commutes with coproducts - and thus in particular with self-coproducts), the result follows. 
\end{proof}

\begin{prop} \label{broad paraskhnio}
    Let $\mathcal{T}$ be a $k$-linear compactly generated triangulated category and $\mathsf{M}$ a separable exact smashing monad on $\mathcal{T}$. If $\mathcal{T}^{\mathsf{b}}=\mathsf{broad}_{\Sigma}(\mathcal{T}^{\mathsf{b}}_c)$ and $\mathsf{M}(\mathcal{T}^{\mathsf{b}}_c)\subseteq \mathcal{T}^{\mathsf{b}}_c$, then 
    \[
    (\mathsf{M}\lMod{_{\mathcal{T}}})^{\mathsf{b}}=\mathsf{broad}_{\Sigma}((\mathsf{M}\lMod{_{\mathcal{T}}})^{\mathsf{b}}_{c})
    \]
as subcategories of $\mathsf{M}\lMod_{\mathcal{T}}$.
\end{prop}
\begin{proof}
    We have the following equalities 
      \begin{align*}
          (\mathsf{M}\lMod{_{\mathcal{T}})^{\mathsf{b}}} & = \mathsf{M}\lMod{_{\mathcal{T}^{\mathsf{b}}}} &   \text{Theorem \ref{main theorem}(i)} \\ 
          & = \mathsf{M}\lMod{_{\mathsf{broad}_{\Sigma}(\mathcal{T}^{\mathsf{b}}_c)}} \\ 
          & = \mathsf{broad}_{\Sigma}(\mathsf{M}\lMod{_{\mathcal{T}^{\mathsf{b}}_c}}) & \text{Lemma \ref{paraskhnio}} \\ 
          & = \mathsf{broad}_{\Sigma}((\mathsf{M}\lMod{_{\mathcal{T}}})^{\mathsf{b}}_c) & \text{Theorem \ref{main theorem}(iv)}
      \end{align*}
      where the first one uses the inclusion $\mathsf{M}(\mathcal{T}^{\mathsf{b}})\subseteq \mathcal{T}^{\mathsf{b}}$, which can be seen as follows 
      \[
      \mathsf{M}(\mathcal{T}^{\mathsf{b}})=\mathsf{M}(\mathsf{broad}_{\Sigma}\mathcal{T}^{\mathsf{b}}_c)\subseteq \mathsf{thick}(\mathsf{M}(\cup_{x\in\mathcal{T}^{\mathsf{b}}_c}\Add(x)))\subseteq \mathsf{thick}(\cup_{x\in\mathcal{T}^{\mathsf{b}}_c}\Add(x))=\mathcal{T}^{\mathsf{b}}. \qedhere
      \]
\end{proof}

\subsection{Comparing homological invariants for separable extensions} \label{Comparing homological invariants for separable extensions} It is due time to prove our first main theorem.

\begin{thm} \label{comparing regularity} 
    Let $\mathcal{T}$ be a compactly generated triangulated category and $\mathsf{M}$ a separable exact smashing monad. The following hold for $\mathcal{D}\coloneqq \mathsf{M}\lMod_{\mathcal{T}}$. 
    \begin{itemize}
        \item[\textnormal{(i)}] Assume that $\mathsf{M}(\mathcal{T}^{\mathsf{b}})\subseteq \mathcal{T}^{\mathsf{b}}$. If $\mathcal{T}$ has finite global dimension, then $\mathcal{D}$ has finite global dimension. 
        \item[\textnormal{(ii)}] Assume that $\mathsf{M}(\mathcal{T}^{\mathsf{b}})\subseteq \mathcal{T}^{\mathsf{b}}$ and $\mathsf{U}_{\mathsf{M}}(\mathcal{D}^{\mathsf{b}}_p)\subseteq \mathcal{T}^{\mathsf{b}}_p$. If $\mathcal{T}$ is Gorenstein, then $\mathcal{D}$ is Gorenstein. 
        \item[\textnormal{(iii)}] Assume that  $\mathcal{T}$ is $k$-linear and $\mathsf{M}(\mathcal{T}^{\mathsf{c}})\subseteq \mathcal{T}^{\mathsf{c}}$. If $\mathcal{T}$ is regular, then $\mathcal{D}$ is regular. 
    \end{itemize}
\end{thm}
\begin{proof}
    (i) An object $(x,u)$ of $\mathcal{D}$ belongs in $\mathcal{D}^{\mathsf{b}}$ if and only if $x\in\mathcal{T}^{\mathsf{b}}$, by Theorem \ref{main theorem}(i). Since $\mathcal{T}$ is assumed to have finite global dimension, the latter is equivalent to $x$ belonging in $\mathcal{T}^{\mathsf{b}}_i$, which is equivalent to $(x,u)\in\mathcal{D}^{\mathsf{b}}_i$ by Theorem \ref{main theorem}(ii) and the assumption. 

    (ii) Consider an object $(x,u)$ of $\mathcal{D}$. If the latter belongs in $\mathcal{D}^{\mathsf{b}}_p$, then it follows from the assumption that $x$ belongs in $\mathcal{T}^{\mathsf{b}}_p$ which coincides with $\mathcal{T}^{\mathsf{b}}_i$ by the assumption. Therefore, it follows from Theorem~\ref{main theorem}(ii) that $(x,u)$ belongs in $\mathcal{D}^{\mathsf{b}}_i$. Similarly, if $(x,u)$ belongs in $\mathcal{D}^{\mathsf{b}}_i$, then it follows from Theorem \ref{main theorem}(ii) that $x$ belongs in $\mathcal{T}^{\mathsf{b}}_i=\mathcal{T}^{\mathsf{b}}_p$ and as a consequence of Theorem \ref{main theorem}(iii), $(x,u)$ belongs in $\mathcal{D}^{\mathsf{b}}_p$.

    (iii) By the assumption $\mathsf{M}(\mathcal{T}^{\mathsf{c}})\subseteq \mathcal{T}^{\mathsf{c}}$ and Theorem \ref{balmer2}, it follows that an object $(x,u)$ of $\mathcal{D}$ is compact if and only if $x$ is compact in $\mathcal{T}$. The claim then follows as in (i), using Theorem \ref{main theorem}(iv). 
\end{proof}

Under an additional assumption on the monad, the converse holds. 

\begin{prop} \label{comparing regularity 2}
    Let $\mathcal{T}$ be a compactly generated triangulated category and $\mathsf{M}$ a separable exact smashing monad such that every object $x$ of $\mathcal{T}$ is a direct summand of $\mathsf{M}(x)$. The following hold for $\mathcal{D}\coloneqq\mathsf{M}\lMod_{\mathcal{T}}$. 
    \begin{itemize}
        \item[\textnormal{(i)}] Assume that the inclusion $\mathsf{M}(\mathcal{T}^{\mathsf{b}})\subseteq \mathcal{T}^{\mathsf{b}}$ holds. If $\mathcal{D}$ has finite global dimension, then $\mathcal{T}$ has finite global dimension.
        \item[\textnormal{(ii)}] Assume that the inclusions $\mathsf{M}(\mathcal{T}^{\mathsf{b}})\subseteq\mathcal{T}^{\mathsf{b}}$, $\mathsf{M}(\mathcal{T}^{\mathsf{b}}_i)\subseteq\mathcal{T}^{\mathsf{b}}_i$ and $\mathsf{U}_{\mathsf{M}}(\mathcal{D}^{\mathsf{b}}_p)\subseteq\mathcal{T}^{\mathsf{b}}_p$ hold. If $\mathcal{D}$ is Gorenstein, then $\mathcal{T}$ is Gorenstein. 
        \item[\textnormal{(iii)}] Assume that $\mathcal{T}$ is $k$-linear and that the inclusions $\mathsf{M}(\mathcal{T}^{\mathsf{c}})\subseteq\mathcal{T}^{\mathsf{c}}$ and $\mathsf{M}(\mathcal{T}^{\mathsf{b}}_c)\subseteq\mathcal{T}^{\mathsf{b}}_c$ hold. If $\mathcal{D}$ is regular, then $\mathcal{T}$ is regular. 
    \end{itemize}
\end{prop}
\begin{proof}
    (i) Given an object $x\in\mathcal{T}^{\mathsf{b}}$, it follows from Theorem \ref{main theorem}(i) and the assumption that $\mathsf{T}_{\mathsf{M}}(x)\in\mathcal{D}^{\mathsf{b}}$. Since $\mathcal{D}$ is assumed to have finite global dimension, $x$ belongs in $\mathcal{D}^{\mathsf{b}}_i$, and therefore $\mathsf{U}_{\mathsf{M}}\mathsf{T}_{\mathsf{M}}(x)=\mathsf{M}(x)$ belongs in $\mathcal{T}^{\mathsf{b}}_i$ by Theorem \ref{main theorem}(ii). Since $x$ is a summand of $\mathsf{M}(x)$, it follows that $x$ itself belongs in $\mathcal{T}^{\mathsf{b}}_i$. Similarly, if we assume that $x$ belongs in $\mathcal{T}^{\mathsf{b}}_i$, then $\mathsf{M}(x)$ belongs in $\mathcal{T}^{\mathsf{b}}$ and has $x$ as a summand. 

    (ii) If $x\in\mathcal{T}^{\mathsf{b}}_p$, then it follows (as in the proof of Theorem \ref{main theorem}(iii)) that $\mathsf{T}_{\mathsf{M}}(x)$ belongs in $\mathcal{D}^{\mathsf{b}}_p=\mathcal{D}^{\mathsf{b}}_i$. By Theorem \ref{main theorem}(ii) and the inclusion $\mathsf{M}(\mathcal{T}^{\mathsf{b}}_i)\subseteq \mathcal{T}^{\mathsf{b}}_i$, it follows that $\mathsf{M}(x)=\mathsf{U}_{\mathsf{M}}\mathsf{T}_{\mathsf{M}}(x)$ belongs in $\mathcal{T}^{\mathsf{b}}_i$ and since $x$ is a summand of $\mathsf{M}(x)$, we infer that $x$ itself belongs in $\mathcal{T}^{\mathsf{b}}_i$, which settles the inclusion $\mathcal{T}^{\mathsf{b}}_p\subseteq \mathcal{T}^{\mathsf{b}}_i$. For the opposite inclusion, given an object $x\in\mathcal{T}^{\mathsf{b}}_i$, it follows from Theorem \ref{main theorem}(ii) that $\mathsf{T}_{\mathsf{M}}(x)$ belongs in $\mathcal{D}^{\mathsf{b}}_i=\mathcal{D}^{\mathsf{b}}_p$. Therefore, it follows that $\mathsf{M}(x)=\mathsf{U}_{\mathsf{M}}\mathsf{T}_{\mathsf{M}}(x)$ belongs in $\mathcal{T}^{\mathsf{b}}_p$ and, as before, we infer that $x\in\mathcal{T}^{\mathsf{b}}_p$. 

    (iii) Given a compact object $x$ of $\mathcal{T}$, it follows that $\mathsf{T}_{\mathsf{M}}(x)$ belongs in $\mathcal{D}^{\mathsf{c}}=\mathcal{D}^{\mathsf{b}}_c$ and therefore we infer from Theorem \ref{main theorem}(iv) that $\mathsf{M}(x)=\mathsf{U}_{\mathsf{M}}\mathsf{T}_{\mathsf{M}}(x)$ belongs in $\mathcal{T}^{\mathsf{b}}_c$, thus $x$ itself belongs in $\mathcal{T}^{\mathsf{b}}_c$ - being a summand of $\mathsf{M}(x)$. For the other inclusion, consider an object $x$ in $\mathcal{T}^{\mathsf{b}}_c$. Then it follows from Theorem \ref{main theorem}(iv) and the inclusion $\mathsf{M}(\mathcal{T}^{\mathsf{b}}_c)\subseteq \mathcal{T}^{\mathsf{b}}_c$ that $\mathsf{T}_{\mathsf{M}}(x)$ belongs in $\mathcal{D}^{\mathsf{b}}_c=\mathcal{D}^{\mathsf{c}}$. Since $\mathsf{M}$ is compact preserving, so is $\mathsf{U}_{\mathsf{M}}$ and therefore $\mathsf{M}(x)=\mathsf{U}_{\mathsf{M}}\mathsf{T}_{\mathsf{M}}(x)$ belongs in $\mathcal{T}^{\mathsf{c}}$, implying also that $x$ is compact.
\end{proof}

\section{Working with Examples}   \label{Working with the examples}

\subsection{Equivariant triangulated categories} \label{Equivariant triangulated categories} Let $\mathcal{T}$ be a triangulated category with coproducts and assume an action of a finite group $G$ on $\mathcal{T}$ such that $|G|$ is invertible in $\mathcal{T}$. As discussed in \ref{equivariant triangulated categories}, there is an adjoint triple of triangle functors $(\mathsf{Ind},\mathsf{F},\mathsf{Ind})$ and the composite $\mathsf{F}\circ \mathsf{Ind}\eqqcolon\mathsf{M}$ is monadic, meaning that the comparison functor $\mathcal{T}^{\mathsf{G}}\rightarrow\mathsf{M}\lMod_{\mathcal{T}}$ is an equivalence. Using the ambidextrous adjunction above, we see that the functors $\mathsf{F}$ and $\mathsf{Ind}$ restrict to the subcategories of compact objects (being left adjoints to coproduct preserving functors). Moreover, we can use Lemma~\ref{perp under adjunction} to deduce that the functors $\mathsf{F}$ and  $\mathsf{Ind}$ restrict to any of the subcategories from Definition \ref{def:subcategories}. As a consequence, the monad $\mathsf{M}=\mathsf{F}\circ \mathsf{Ind}$ is smashing, compact preserving and restricts to any of the subcategories from Definition \ref{def:subcategories}. Moreover, since the autoequivalences $\rho_g\colon \mathcal{T}\rightarrow \mathcal{T}$, that define the action of $G$, restrict to autoequivalences $\rho_g\colon \mathcal{T}^{\mathsf{b}}_?\rightarrow \mathcal{T}^{\mathsf{b}}_?$ for any $?\in\{\varnothing, i,p,c\}$ (where for $?=c$ we assume that $\mathcal{T}$ is $k$-linear), it follows that the action of $G$ on $\mathcal{T}$ restricts to an action on the latter subcategories. 

\begin{cor} \label{subcategories for equivariant}
Let $\mathcal{T}$ be a compactly generated triangulated category and assume an action of a finite group $G$ on $\mathcal{T}$ such that $|G|$ is invertible in $\mathcal{T}$. The following hold. 
    \begin{itemize}
        \item[\textnormal{(i)}] $(\mathcal{T}^{\mathsf{c}})^G=(\mathcal{T}^G)^{\mathsf{c}}$;
        \item[\textnormal{(ii)}] $(\mathcal{T}^{\mathsf{b}})^G= (\mathcal{T}^G)^{\mathsf{b}}$; 
        \item[\textnormal{(iii)}] $(\mathcal{T}^{\mathsf{b}}_i)^G= (\mathcal{T}^G)^{\mathsf{b}}_i$; 
        \item[\textnormal{(iv)}] $(\mathcal{T}^{\mathsf{b}}_p)^G= (\mathcal{T}^G)^{\mathsf{b}}_p$;

        \item[\textnormal{(v)}]
        $(\mathcal{T}^{\mathsf{b}}_c)^G= (\mathcal{T}^G)^{\mathsf{b}}_c$ \textnormal{(}when $\mathcal{T}$ is $k$-linear\textnormal{)}.
    \end{itemize}
    as triangulated subcategories of $\mathcal{T}^G$. 
\end{cor}
\begin{proof}
    Statement (i) was already proved in Proposition \ref{compactly generated equivariant} and the rest follow from Theorem \ref{main theorem}.
\end{proof}

Note that $x \in \mathcal{T}$ is a summand of $\mathsf{M}(x)$, since $\mathsf{M}(x) = \oplus \rho_g(x)$, and therefore the condition of Proposition~\ref{comparing regularity 2} is satisfied. We thus obtain the following result.

\begin{cor} \label{regularity for equivariant}
    Assume the setup of Corollary \ref{subcategories for equivariant} and consider the following conditions. 
    \begin{itemize} 
        \item[\textnormal{(i)}] $\mathcal{T}$ has finite global dimension.
        \item[\textnormal{(ii)}] $\mathcal{T}$ is Gorenstein. 
        \item[\textnormal{(iii)}] $\mathcal{T}$ is regular \textnormal{(}when $\mathcal{T}$ is $k$-linear\textnormal{)}.
    \end{itemize}
    Then $\mathcal{T}$ satisfies any of \textnormal{(i)-(iii)} if and only if $\mathcal{T}^G$ satisfies the same condition.
\end{cor}
\begin{proof}
    This is a consequence of Theorem \ref{comparing regularity} and Proposition \ref{comparing regularity 2}. 
\end{proof}

\subsubsection{Example 1: The derived category of a ring} Given an action of a finite group $G$ on a ring $R$, there is an induced action on $\Mod R$ and a ring denoted by $R\#G$, called the \emph{skew group ring}, which satisfy
\[
\Mod R\#G\simeq (\Mod R)^G,
\]
see \cite{reiten_riedtmann} for details. As a consequence of Theorem \ref{equivariant derived}, when $|G|$ is invertible in $R$, it follows that 
\[
\mathsf{D}(R)^G\simeq \mathsf{D}(R\#G)
\]
since $\mathsf{D}(R)$ is idempotent complete. We can therefore obtain the following well-known result of \cite{reiten_riedtmann} as an immediate consequence of Corollary \ref{regularity for equivariant}. 

\begin{cor}  \label{regularity for the derives category}
Let $R$ be a ring and $G$ a finite group acting on $R$ such that $|G|$ is invertible in $R$. The following hold. 
\begin{itemize}
    \item[\textnormal{(i)}] $\gd R<\infty$ if and only if $\gd R\#G<\infty$. 
    \item[\textnormal{(ii)}] $R$ is Gorenstein if and only if $R\#G$ is Gorenstein. 

    \item[\textnormal{(iii)}] When $R$ is a Noether algebra, then $R$ is regular if and only if $R\#G$ is regular. 
\end{itemize}
\end{cor}

We end this section with the following result, which was promised above Lemma \ref{paraskhnio}. 

\begin{cor} \label{perissos}
    Let $R$ be a commutative noetherian ring and $G$ a finite group acting on $R$ such that $|G|$ is invertible in $R$. If condition \textnormal{(\ref{paraskhnio2})} holds for $R$, then it also holds for $R\#G$.
\end{cor}
\begin{proof}
    It follows from Proposition \ref{broad paraskhnio} that, for $\mathcal{T}=\mathsf{D}(R)$, we have $(\mathcal{T}^{\mathsf{G}})^{\mathsf{b}}=\mathsf{broad}_{\Sigma}((\mathcal{T}^{G})^{\mathsf{b}}_c)$. Since $G$ is finite, $R\#G$ is a Noether $R$-algebra and therefore $(\mathcal{T}^G)^{\mathsf{b}}_c=\mathsf{D}^{\mathsf{b}}(\smod R\#G)$, which completes the claim. 
\end{proof}

We know from \cite[Subsection 3.4]{regular} that condition (\ref{paraskhnio2}) holds for every commutative noetherian quasiexcellent ring of finite Krull dimension. As a consequence, the same condition holds for a big class of skew group algebras over such rings, which are typically noncommutative, providing new examples. In particular, for such rings, the computation of the intrinsic subcategories for the homotopy category of injectives, as in \cite[Proposition 3.21]{regular}, hold.

\subsubsection{Example 2: The derived category of a scheme} \label{action on scheme} Let $X$ be a proper scheme over a commutative noetherian ring $k$ and consider $\mathcal{T}=\mathsf{D}_{\mathsf{qc}}(X)$ the derived category of $\mathcal{O}_X$-modules with quasi-coherent cohomology. We then know that $\mathcal{T}$ is compactly generated and $\mathcal{T}^{\mathsf{c}}=\mathsf{D}^{\mathsf{perf}}(X)$, see \cite[Theorem 3.1.1]{bondal_van-den-bergh}. Moreover, it follows by \cite[Theorem 8.20]{neeman7} (see in particular \cite[Example 0.7]{neeman7}) that $\mathcal{T}^{\mathsf{b}}_c=\mathsf{D}^{\mathsf{b}}_{\mathsf{coh}}(X)$. It is known that $X$ is regular if and only if $\mathsf{D}^{\mathsf{perf}}(X)=\mathsf{D}^{\mathsf{b}}_{\mathsf{coh}}(X)$. In our language, the latter translates to the following. 
\begin{cor}
    A scheme $X$, with the above assumptions, is regular if and only if $\mathsf{D}(\mathsf{Qcoh}X)$ is regular. 
\end{cor}
\begin{proof}
     By \cite[Corollary 5.5]{bokstedt_neeman} there is an equivalence $\mathsf{D}_{\mathsf{qc}}(X)\simeq \mathsf{D}(\mathsf{Qcoh}X)$, so the claim follows from the above discussion. 
\end{proof}
Note that, by \cite[Proposition~0FDB]{StacksProject}, the equivalence in the proof above restricts to an equivalence $\mathsf{D}_{\mathsf{coh}}^{\mathsf{b}}(X)\simeq \mathsf{D}^{\mathsf{b}}(\mathsf{coh}(X))$, implying that for $\mathcal{T}=\mathsf{D}(\mathsf{Qcoh}(X))$ we have that $\mathcal{T}^{\mathsf{b}}_c=\mathsf{D}^{\mathsf{b}}(\mathsf{coh}(X))$. 

Assume now an action of a finite group $G$ on $X$. The quotient scheme $X/G$ exists if and only if each orbit is contained in some affine open, see \cite[Expos\'e V, Proposition 1.8]{SGA1}. The action on $X$ induces a canonical action on $\mathsf{Qcoh}(X)$ and $\mathsf{coh}(X)$ by considering the pushforward induced by each $\rho_g\colon X\rightarrow X$. Thus, we can consider the equivariant categories $\mathsf{Qcoh}^G(X)\coloneqq \mathsf{Qcoh}(X)^G$ and $\mathsf{coh}^G(X)\coloneqq\mathsf{coh}(X)^G$. If the action is assumed, moreover, to be free, then we have equivalences $\mathsf{Qcoh}^G(X) \simeq \mathsf{Qcoh}(X/G)$ and $\mathsf{coh}^G(X)\simeq \mathsf{coh}(X/G)$. The $G$-action on these categories extends naturally to a $G$-action on $\mathsf{D}(\mathsf{Qcoh}(X))$ and $\mathsf{D^b}(\mathsf{coh}(X))$, for which we have canonical equivalences when $|G|$ is invertible in $k$:
\[
    \mathsf{D}(\mathsf{Qcoh}(X))^G \simeq \mathsf{D}(\mathsf{Qcoh}^G(X)) \simeq \mathsf{D}(\mathsf{Qcoh}(X/G)) \ \ \ \& \ \ \ \mathsf{D^b}(\mathsf{coh}(X))^G \simeq \mathsf{D^b}(\mathsf{coh}^G(X)) \simeq \mathsf{D^b}(\mathsf{coh}(X/G)).
\]
The first equivalences in each pair follows from Theorem~\ref{equivariant derived}, since both $\mathsf{D}(\mathsf{Qcoh}(X))$ and $\mathsf{D^b}(\mathsf{coh}(X))$ are idempotent complete. Indeed, $\mathsf{Qcoh}(X)$ is Grothendieck, by \cite[Proposition~077P]{StacksProject}), hence it has coproducts, thus $\mathsf{Qcoh}^G(X)$ also has coproducts and thus $\mathsf{D}(\mathsf{Qcoh}^G(X))$ is idempotent complete.

\begin{cor} \label{regularity for schemes}
    Let $X$ be a proper scheme over a commutative noetherian ring $k$. Assume a free action of a finite group $G$ on $X$ such that $|G|$ is invertible in $k$.
    If the quotient scheme $X/G$ exists, then $X$ is regular if and only if $X/G$ is regular. In particular, this always holds when $X$ is quasi-projective variety.
\end{cor}

\begin{proof}
We write $\mathcal{T}$ for $ \mathsf{D_{qc}}(X)$ and $\mathcal{D}$ for $ \mathsf{D_{qc}}(X/G)$. Note that $X$ being proper implies that $X/G$ is also proper (in particular noetherian), thus we have the following canonical equivalences
$$\mathcal{D} \simeq  \mathsf{D}(\mathsf{Qcoh}(X/G)) \simeq  \mathsf{D}(\mathsf{Qcoh}^G(X)) \simeq  \mathsf{D}(\mathsf{Qcoh}({X}))^G \simeq \mathcal{T}^G. $$ 
Since $X/G$ is proper, it follows that $\mathcal{D}^{\mathsf{c}}=\mathsf{D}^{\mathsf{perf}}(X/G)$ and $\mathcal{D}^{\mathsf{b}}_c=\mathsf{D}^{\mathsf{b}}(\mathsf{coh}(X/G))$. We can, therefore, use Corollary \ref{subcategories for equivariant} and Corollary \ref{regularity for equivariant} to conclude the first claim. In particular, when $X$ is a quasi-projective variety, then it is proper and the quotient scheme always exists \cite[Chapter II, Paragraph 7, Remark]{Mumford}. 
\end{proof}

\subsubsection{Example 3: The derived category of a dg algebra} \label{the derived category of a dg algebra} Let $A$ be a dg category over a field $k$ and $G$ a finite group which acts strongly on $A$, i.e.\ via dg automorphisms. From now on, whenever we say that a group acts on a dg category, we mean a strong action. In this case the equivariant category $A^G$ is also a dg category, see \cite[Proposition 8.3]{elagin}. We recall the following relevant construction following \cite[Subsection~4.8]{heisenberg} (see also \cite[Subsection~3.1]{MerlinChrist}). 

\begin{defn}
    In the above setup, the \emph{skew group dg category}, denoted by $A\#G$, is defined as follows.    \begin{itemize}
        \item[$\bullet$] $\mathsf{Ob}A\#G=\mathsf{Ob}A$. 
        \item[$\bullet$] For any objects $a,a'$ of $A\#G$, 
        \[
        \mathsf{Hom}_{A\#G}^i(a,a')\coloneqq \{(f,g) \ | \ f\in\mathsf{Hom}_{A}^i(ga,a'), g\in G\}
        \]
        and moreover $\mathsf{deg}(f,g)\coloneqq \mathsf{deg}(f)$ and $\mathsf{d}(f,g)\coloneqq (\mathsf{d}f,g)$. 
        \item The composition in $A\#G$ is defined by 
        \[
        (f,g)\circ (f',g')\coloneqq (f\circ gf',gg'). 
        \]
        \item For any $a\in A\#G$, the identity morphism is $(\mathsf{Id}_a,1)$. 
    \end{itemize}
\end{defn} 

The following is a consequence of the definition. 

\begin{lem} \label{connective and proper}
    Let $A$ be a dg algebra and $G$ a finite group acting on $A$. The following hold. 
    \begin{itemize}
        \item[\textnormal{(i)}] $A\#G$ is a dg algebra. 
        \item[\textnormal{(ii)}] If $A$ is proper \textnormal{(}i.e.\ $\oplus_{i\in\mathbb{Z}}H^i(A)\in\smod k$\textnormal{)}, then $A\#G$ is also proper. Similarly, if $A$ is connective \textnormal{(}i.e.\ $H^i(A)=0$ for $i>0$\textnormal{)}, then $A\#G$ is connective.
    \end{itemize}
\end{lem}

An action of $G$ on $A$ as above induces an action of $G$ on $\mathsf{D}(A)$, which in fact has the following property (we refer to \cite{keller0} for the definition of the derived category of a dg category). 

\begin{prop} \label{derived category of dg skew}
    Let $A$ be a dg category over a field $k$ and $G$ a finite group such that $|G|$ is invertible in $k$. Assume an action of $G$ on $A$. Then there is an equivalence of triangulated categories 
    \[
    \mathsf{D}(A\#G)\simeq \mathsf{D}(A)^G.
    \]
\end{prop}
\begin{proof}
We consider the morphism of dg categories $\phi\colon A\rightarrow A\#G$ given by the identity on objects and for every $a,a'\in A$, $\mathsf{Hom}_A^i(a,a')\rightarrow \mathsf{Hom}_{A\#G}^i(a,a')$ maps $f$ to $(f,1_G)$. This gives rise to a restriction functor $\phi^*\colon \dgMod A\#G\rightarrow \dgMod A$ on the level of dg module categories. The action of $G$ on $A$ induces an action of $G$ on $\dgMod A$ by considering the induction functors $(\rho_g)_*\colon \dgMod A\rightarrow \dgMod A$. We know from \cite[Lemma~4.41]{heisenberg} that there is an equivalence of dg categories
\begin{equation} \label{equivalence from heisenberg}
    \dgMod A\#G\simeq (\dgMod A)^G
\end{equation}
and under this equivalence $\phi^*$ commutes with the forgetful functor $(\dgMod A)^G\rightarrow \dgMod A$ (note that in \emph{loc.\! cit.\ }this is proved for a right action). As a consequence, the functor $\phi^*$ is a part of an ambidextrous adjunction $(\phi_*,\phi^*,\phi_*)$ and it holds that 
\begin{equation} \label{M2}
    \phi_*\phi^*=\oplus_{g\in G}(\rho_g)_*.
\end{equation}
Moreover, there exist natural transformations 
\[
\eta\colon \mathsf{Id}_{\dgMod A\#G}\rightarrow \phi_*\phi^* \ \ \ \ \ \& \ \ \ \ \  \epsilon \colon\phi_*\phi^*\rightarrow \mathsf{Id}_{\dgMod A\#G}
\]
where $\eta$ is the unit of the adjunction $(\phi^*,\phi_*)$ and $\epsilon$ is the counit of the adjunction $(\phi_*,\phi^*)$. The latter satisfy $\epsilon \eta=|G|$. This adjunction gives rise to an adjunction on the level of homotopy categories, on the nose, and it is directly verified that $\phi^*$ and $\phi_*$ preserve acyclic modules. We can therefore deduce the existence of an adjunction of derived categories as follows. 
\[
\begin{tikzcd}
\mathsf{D}(A\#G) \arrow[rr,  "\phi^*_{\phantom{*}}" description] &  & \mathsf{D}(A) \arrow[ll, "\ \ \phi_*^{\phantom{*}}"', bend right] \arrow[ll, "\ \  \phi_*^{\phantom{*}}", bend left]
\end{tikzcd}
\]
The unit and counit of the adjunction above, which we again denote by $\eta$ and $\epsilon$ respectively, satisfy the same relation as before. In particular, the unit $\eta$ has a left inverse (namely $\frac{\epsilon}{|G|}$). It follows, see for instance \cite[Lemma~2.2]{xwchen}, that the functor $\phi^*$ is separable, and therefore there is an exact equivalence
\[
\mathsf{D}(A\#G)\simeq \mathsf{M}\lMod{_{\mathsf{D}(A)}}
\]
where $\mathsf{M}$ is the monad $\phi^*\phi_*$. The action of $G$ on $\dgMod A$ induces an action of $G$ on $\mathsf{D}(A)$. Since $\phi_*\phi^*=\oplus_{g\in G}(\rho_g)_*$ on the level of dg module categories, by (\ref{M2}), it follows that $\mathsf{M}$ coincides with $\oplus_{g\in G}(\rho_g)_*$ on the level of the derived category $\mathsf{D}(A)$. Since the latter is precisely the monad induced by the action of $G$ on $\mathsf{D}(A)$, the proof is complete.
\end{proof}

For any dg algebra $A$ over a field $k$ the subcategory of $\mathsf{D}(A)$ that consists of the finite dimensional modules is $\mathsf{D}_{\mathsf{fd}}(A)\coloneqq \{x\in\mathsf{D}(A): \oplus_{n\in\mathbb{Z}} H^n(x)\in\smod k\}$.
This clearly coincides with  $\mathcal{T}^{\mathsf{b}}_c$ (for $\mathcal{T}=\mathsf{D}(A)$), since $A$ is a compact generator of $\mathsf{D}(A)$. Recall also that $A$ is called \emph{smooth} if $A\in\mathsf{per}(A\otimes_kA^{\mathrm{op}})$. Moreover, according to \cite{jin}, we say that $A$ is Gorenstein if $\mathsf{thick}(A)=\mathsf{thick}(A^*)$, where $A^*=\mathbb{R}\mathsf{Hom}_k(A,k)$.

\begin{prop} \label{smoothness vs finite global dim}
    Let $A$ be a proper connective dg algebra over a field $k$. The following hold. 
    \begin{itemize}
        \item[\textnormal{(i)}] If $k$ is perfect, then $A$ is smooth if and only if $\mathsf{D}(A)$ is regular if and only if $\mathsf{D}(A)$ has finite global dimension. 
        \item[\textnormal{(ii)}] $A$ is Gorenstein if and only if $\mathsf{D}(A)$ is Gorenstein. 
    \end{itemize}
\end{prop}

\begin{proof}
    See \cite[Corollary 5.15]{regular} and  \cite[Proposition 5.6]{regular} respectively.
\end{proof}

\begin{cor} \label{homological invariants for skew group dg algebras}
    Let $A$ be a proper connective dg algebra over a field $k$. Assume an action of a finite group $G$ on $A$ such that $|G|$ is invertible in $k$. The following hold. 
    \begin{itemize}
        \item[\textnormal{(i)}] If $k$ is perfect, then $A$ is smooth if and only if $A\#G$ is smooth. 
        \item[\textnormal{(ii)}] $A$ is Gorenstein if and only if $A\#G$ is Gorenstein. 
    \end{itemize}
\end{cor}
\begin{proof}
This is a direct consequence of Corollary \ref{regularity for equivariant} together with Proposition \ref{smoothness vs finite global dim} and Lemma \ref{connective and proper}.
\end{proof}

\subsection{Separable extensions} \label{Separable extensions} Let $A$ be a flat separable $R$-algebra, where $R$ is a commutative ring. The subcategories of $\mathsf{D}(A)$ are known in full generality, so we cannot use Theorem \ref{main theorem} to produce anything useful. Nonetheless, Theorem \ref{comparing regularity} applied to the equivalence of Theorem \ref{etale Balmer} recovers some known results (see for instance \cite{bhatt, li}) in an interesting way. 

\begin{cor}
    Consider $A$ and $R$ as above. The following hold. 
    \begin{itemize}
        \item[\textnormal{(i)}] If $R$ has finite global dimension, then $A$ has finite global dimension. 
        \item[\textnormal{(ii)}] If $R$ is Gorenstein and $A$ is finitely generated \textnormal{(}as an $R$-module\textnormal{)}, then $A$ is Gorenstein. 
        \item[\textnormal{(iii)}] If $R$ is regular noetherian and $A$ is finitely generated \textnormal{(}as an $R$-module\textnormal{)}, then $A$ is regular. 
    \end{itemize}
\end{cor}
\begin{proof}
    We know from Theorem \ref{etale Balmer} that there is an exact equivalence $\mathsf{D}(A)\simeq \mathsf{M}\lMod_{\mathsf{D}(R)}$ where the monad $\mathsf{M}$ is given by $-\otimes_RA$. Since $A$ is a flat $R$-module, the functor $\mathsf{M}$ restricts to a functor between complexes with bounded cohomology, which shows (i), using Theorem \ref{comparing regularity}. When $R$ is noetherian, then $A$ is projective as an $R$-module (being flat and finitely generated) and therefore the functor $\mathsf{U}_{\mathsf{M}}$ maps $\mathsf{K}^{\mathsf{b}}(\Proj A)$ to $\mathsf{K}^{\mathsf{b}}(\Proj R)$ and the functor $\mathsf{M}$ restricts to compact objects, which shows (ii) and (iii) respectively - again using Theorem \ref{comparing regularity}.
\end{proof}

Under an additional hypothesis, we can verify the assumption of Proposition \ref{comparing regularity 2} (and in particular conclude the converse to the implications of the above result).

\begin{lem} \label{split extension}
    Let $A$ be a separable flat $R$-algebra. If the ring homomorphism $R\to A$ admits a section, then every object $x$ of $\mathsf{D}(R)$ is a summand of $\mathsf{M}(x)=x\otimes_RA$. 
\end{lem}
\begin{proof}
    There exist functors as in the following diagram 
    \[
\begin{tikzcd}
\mathsf{D}(R) \arrow[rr, "\mathsf{i}"] &  & \mathsf{D}(A) \arrow[rr, "\mathsf{e}"] &  & \mathsf{D}(R) \arrow[ll, "\mathsf{l}=-\otimes_{R}A"', bend right] \arrow["\mathsf{M}=\mathsf{el}"', loop, distance=2em, in=125, out=55]
\end{tikzcd}
    \]
    where the functor $\mathsf{i}$ is the restriction induced by the section $A\rightarrow R$. Consider an object $x$ in $\mathsf{D}(R)$. We know from Theorem \ref{etale Balmer}, together with Lemma \ref{summand}, that $\mathsf{i}(x)$ is a summand of $\mathsf{le}(\mathsf{i}(x))$. Consequently $\mathsf{ei}(x)\cong x$ is a summand of $\mathsf{elei}(x)\cong \mathsf{el}(x)\cong \mathsf{M}(x)$. 
\end{proof}

One direction of Corollary \ref{regularity for schemes} can be generalised to the following result, which is well-known (see for instance \cite[Proposition 3.17]{milne}), but follows as an easy consequence of our results. 

\begin{cor} \label{regularity for etale maps}
    Let $f\colon Y \rightarrow X$ be a separated \'etale morphism of schemes that are proper over a commutative noetherian ring $k$. If $X$ is regular, then $Y$ is regular. 
\end{cor}
\begin{proof}
    In this case, there is an adjunction of exact functors as below 
\[\begin{tikzcd}
	{\mathcal{D}\coloneqq\mathsf{D}(\mathsf{Qcoh}(Y))} && {\mathsf{D}(\mathsf{Qcoh}(X))\eqqcolon\mathcal{T}}
	\arrow["{\mathbb{R}f_*}"', shift right=2, from=1-1, to=1-3]
	\arrow["{f^*}"', shift right=2, from=1-3, to=1-1]
\end{tikzcd}\]
    and we know from \cite[3.5 Theorem]{balmer2} that it induces an exact equivalence $\mathcal{D}\simeq \mathbb{A}_f\lMod_{\mathcal{T}}$, where the monad $\mathbb{A}_f\coloneqq \mathbb{R}f_*\circ f^*$ is smashing. The functor $f^*$ restricts to a functor $\mathsf{D}^{\mathsf{b}}(\mathsf{coh}(X))\rightarrow \mathsf{D}^{\mathsf{b}}(\mathsf{coh}(Y))$. Moreover, the same holds for $\mathbb{R}f_*$, namely it restricts to $\mathsf{D}^{\mathsf{b}}(\mathsf{coh}(Y))\rightarrow \mathsf{D}^{\mathsf{b}}(\mathsf{coh}(X))$, being right adjoint to the compact-preserving functor $f^*$ (see Lemma \ref{perp under adjunction}). Therefore the monad $\mathbb{A}_f$ restricts to $\mathsf{D}^{\mathsf{b}}(\mathsf{coh}(X))=\mathcal{T}^{\mathsf{b}}_c$. In particular, when $X$ is regular, the monad $\mathbb{A}_f$ restricts to the compact objects and therefore the result follows from Theorem \ref{comparing regularity}. 
\end{proof}

\section{Singularity Categories} \label{singularity categories}
We prove the following result relating the module category of a Verdier quotient and the Verdier quotients of module categories. For equivariant triangulated categories this was shown in \cite[Theorem~3.9]{sun}. We use this as a step in the proof of Theorem B of the introduction, which follows right after.   

\begin{prop} \label{verdier quotients}
    Let $\mathcal{T}$ be a triangulated category and $\mathsf{M}$ a separable exact monad on $\mathcal{T}$. Consider a thick triangulated subcategory $\mathcal{U}$ of $\mathcal{T}$ such that $\mathsf{M}(\mathcal{U})\subseteq \mathcal{U}$. The following hold. 
    \begin{itemize}
        \item[\textnormal{(i)}] $\mathsf{M}$ induces a monad $\overline{\mathsf{M}}$ on $\mathcal{T}/\mathcal{U}$ and $\overline{\mathsf{M}}\lMod_{\mathcal{T}/\mathcal{U}}$ admits a unique canonical triangulated structure. 
        \item[\textnormal{(ii)}] There is an exact functor $\mathsf{M}\lMod_{\mathcal{T}}/\mathsf{M}\lMod_{\mathcal{U}}\rightarrow \overline{\mathsf{M}}\lMod_{\mathcal{T}/\mathcal{U}}$ that is an equivalence up to retracts. 
    \end{itemize}
\end{prop}
\begin{proof}
    Using the assumption that $\mathsf{M}(\mathcal{U})\subseteq \mathcal{U}$, we infer the existence of the following commutative diagram of exact functors. 
\[\begin{tikzcd}
    {\mathcal{D}\coloneqq\mathsf{M}\lMod_{\mathcal{T}}} && {\mathcal{T}} \\
	\\
	{\mathcal{D}'\coloneqq \mathsf{M}\lMod_{\mathcal{U}}} && {\mathcal{U}}
	\arrow["{{{\mathsf{U}_{\mathsf{M}}}}}"', shift right=2, from=1-1, to=1-3]
	\arrow["{{{\mathsf{T}_{\mathsf{M}}}}}"', shift right=2, from=1-3, to=1-1]
	\arrow[hook, from=3-1, to=1-1]
	\arrow["{{{\mathsf{U}_{\mathsf{M}}}}}"', shift right=2, from=3-1, to=3-3]
	\arrow[hook, from=3-3, to=1-3]
	\arrow["{{{\mathsf{T}_{\mathsf{M}}}}}"', shift right=2, from=3-3, to=3-1]
\end{tikzcd}\]
    We may therefore pass to the Verdier quotients and deduce the existence of an adjunction of exact functors
\[\begin{tikzcd}
	{\mathcal{D}/\mathcal{D}'} && {\mathcal{T}/\mathcal{U}}
	\arrow["{{{\overline{\mathsf{U}_{\mathsf{M}}}}}}"', shift right=2, from=1-1, to=1-3]
	\arrow["{{{\overline{\mathsf{T}_{\mathsf{M}}}}}}"', shift right=2, from=1-3, to=1-1].
\end{tikzcd}\]
    The composite $\overline{\mathsf{M}}\coloneqq \overline{\mathsf{U}_{\mathsf{M}}}\circ \overline{\mathsf{T}_{\mathsf{M}}}$ defines a monad on $\mathcal{T}/\mathcal{U}$ (and it is precisely the functor induced by $\mathsf{M}$ after passing to the Verdier quotient $\mathcal{T}/\mathcal{U}$). We know from \cite[Lemma 3.10]{sun} that the functor $\overline{\mathsf{U}_{\mathsf{M}}}$ is separable and therefore, it follows from \cite[Proposition 3.5]{xwchen} that the comparison functor $\mathcal{D}/\mathcal{D}'\rightarrow~ \overline{\mathsf{M}}\lMod_{\mathcal{T}/\mathcal{U}}$ (which is an exact functor - with the triangulated structure given on $\overline{\mathsf{M}}\lMod_{\mathcal{T}/\mathcal{U}}$ from Lemma \ref{pretriangulated without idempotent}) is an equivalence up to retracts. Since the idempotent completion of $\overline{\mathsf{M}}\lMod_{\mathcal{T}/\mathcal{U}}$ is triangulated (compatible with the triangulated structure given on $\overline{\mathsf{M}}\lMod_{\mathcal{T}/\mathcal{U}}$), so is $\overline{\mathsf{M}}\lMod_{\mathcal{T}/\mathcal{U}}$ (observe that this holds independently of our Convention \ref{convention}, as long as $\mathsf{M}\lMod_{\mathcal{T}}$ is canonically triangulated).
\end{proof}

For any subcategories $\mathcal{X}$ and $\mathcal{Y}$ of a triangulated category $\mathcal{T}$, we denote by $\mathcal{X}\ast \mathcal{Y}$ the full subcategory of $\mathcal{T}$ that consists of the objects $t$ such that there is a triangle $x\rightarrow t\rightarrow y\rightarrow x[1]$ for some $x\in\mathcal{X}$ and $y\in\mathcal{Y}$. In \cite{regular}, the \emph{singularity category} of a $k$-linear compactly generated triangulated category $\mathcal{T}$ is defined as the Verdier quotient 
\begin{equation} \label{singularity definition}
    \mathcal{T}^{\mathsf{sg}}_k\coloneqq\mathsf{thick}( \mathcal{T}^{\mathsf{b}}_{c}\ast \mathcal{T}^{\mathsf{c}})/\mathcal{T}^{\mathsf{b}}_c\cap \mathcal{T}^{\mathsf{c}}.
\end{equation}
The above quotient ``measures regularity'' in the sense that $\mathcal{T}$ is regular if and only if $\mathcal{T}^{\mathsf{sg}}_k$ is trivial, see \cite[Lemma~5.8]{regular}.
\begin{exmp} \label{examples of singularity}
    In the following relevant to us examples, $k$ remains a commutative noetherian ring.
    \begin{itemize}
        \item Consider $\mathcal{T}=\mathsf{D}(A)$ the derived category of a Noether $k$-algebra. Then $\mathcal{T}^{\mathsf{sg}}_k$ is $\mathsf{D}_{\mathsf{sg}}(A)$, the usual singularity category of $A$. 
        \item Consider $\mathcal{T}=\mathsf{D}_{\mathsf{qc}}(X)$, where $X$ is a proper scheme over $k$. Then $\mathcal{T}^{\mathsf{sg}}_k$ is $\mathsf{D}_{\mathsf{sg}}(X)$, the usual singularity category of $X$ (see Subsection \ref{action on scheme}). 
        \item Consider $\mathcal{T}=\mathsf{D}(A)$ where $A$ is a proper dg algebra. The properness of $A$ implies the inclusion $\mathsf{per}(A)\subseteq \mathsf{D}_{\mathsf{fd}}(A)=\mathcal{T}^{\mathsf{b}}_c$, and 
        $\mathcal{T}^{\mathsf{sg}}_k$ is $\mathsf{D}_{\mathsf{fd}}(A)/\mathsf{per}(A)$.
    \end{itemize}
\end{exmp}

\begin{thm} \label{singularity}
    Let $\mathcal{T}$ be a $k$-linear compactly generated triangulated category and $\mathsf{M}$ an exact separable smashing monad on $\mathcal{T}$. Assume further that the inclusions $\mathsf{M}(\mathcal{T}^{\mathsf{b}}_c)\subseteq \mathcal{T}^{\mathsf{b}}_c$ and $\mathsf{M}(\mathcal{T}^{\mathsf{c}})\subseteq \mathcal{T}^{\mathsf{c}}$ hold. Then the monad $\mathsf{M}$ induces a monad $\overline{\mathsf{M}}$ on $\mathcal{T}^{\mathsf{sg}}_k$ and there is an exact functor
    \[
    {(\mathsf{M}  \ \!{\lMod}_{\mathcal{T}})}^{\mathsf{sg}}_k\rightarrow \overline{\mathsf{M}} \ \!{\lMod}_{\mathcal{T}^{\mathsf{sg}}_k}
    \]
 which is an equivalence up to retracts. 
\end{thm} 

\begin{proof}
    As always we denote $\mathsf{M}\lMod_{\mathcal{T}}$ by $\mathcal{D}$. We claim first that the adjunction $(\mathsf{T}_{\mathsf{M}},\mathsf{U}_{\mathsf{M}})$ restricts to adjunctions of full subcategories as in the following diagram. 
    \[
    \begin{tikzcd}
	{\mathsf{thick}(\mathcal{D}^{\mathsf{b}}_c\ast \mathcal{D}^{\mathsf{c}})} && {\mathsf{thick}(\mathcal{T}^{\mathsf{b}}_c\ast \mathcal{T}^{\mathsf{c}})} \\
	\\
	{\mathcal{D}^{\mathsf{b}}_c\cap \mathcal{D}^{\mathsf{c}}} && {\mathcal{T}^{\mathsf{b}}_c\cap\mathcal{T}^{\mathsf{c}}}
	\arrow["{\mathsf{U}_{\mathsf{M}}}"', shift right=2, from=1-1, to=1-3]
	\arrow["{\mathsf{T}_{\mathsf{M}}}"', shift right=2, from=1-3, to=1-1]
	\arrow[hook, from=3-1, to=1-1]
	\arrow["{\mathsf{U}_{\mathsf{M}}}"', shift right=2, from=3-1, to=3-3]
	\arrow[hook, from=3-3, to=1-3]
	\arrow["{\mathsf{T}_{\mathsf{M}}}"', shift right=2, from=3-3, to=3-1]
\end{tikzcd}
    \]
    Indeed, regarding $\mathsf{U}_{\mathsf{M}}$, it always restricts to the subcategories of bounded finite objects and it also restricts to compact objects because of the inclusion $\mathsf{M}(\mathcal{T}^{\mathsf{c}})\subseteq \mathcal{T}^{\mathsf{c}}$. Regarding $\mathsf{T}_{\mathsf{M}}$, it always preserves compact objects and it also preserves bounded finite objects because of the inclusion $\mathsf{M}(\mathcal{T}^{\mathsf{b}}_c)\subseteq \mathcal{T}^{\mathsf{b}}_c$ combined with Theorem \ref{main theorem}(iv). From the above one can also easily deduce the functors of the diagram. Since $\mathsf{thick}(\mathcal{D}^{\mathsf{b}}_c\ast \mathcal{D}^{\mathsf{c}})$ and $\mathcal{D}^{\mathsf{b}}_c\cap \mathcal{D}^{\mathsf{c}}$ are both thick, we can identify them with $\mathsf{M}\lMod_{\mathsf{thick}(\mathcal{T}^{\mathsf{b}}_c\ast \mathcal{T}^{\mathsf{c}})}$ and $\mathsf{M}\lMod_{\mathcal{T}^{\mathsf{b}}_c\cap \mathcal{T}^{\mathsf{c}}}$ respectively, using Lemma \ref{underdog}. In particular, we can use Proposition \ref{verdier quotients} to conclude the result.  
  \end{proof}

\begin{exmp} \label{singularity of separable}
Let $A$ be a Noether algebra over $k$. Assume further that $A$ is flat and separable (for instance a commutative noetherian ring that is \'etale over $k$). Then, for $\mathsf{M}=-\otimes_kA\colon\mathsf{D}_{\mathsf{sg}}(k)\rightarrow \mathsf{D}_{\mathsf{sg}}(k)$, it follows from Theorem \ref{singularity} together with Example \ref{examples of singularity} that there is an exact functor 
    \[
    \mathsf{D}_{\mathsf{sg}}(A)\rightarrow \mathsf{M} \ \! {\lMod}_{\mathsf{D}_{\mathsf{sg}}(k)}
    \]
    which is an equivalence up to retracts. If $A$ is assumed to be an Artin algebra, then the above is an equivalence, since in this case the singularity category is idempotent complete, by \cite[Corollary~2.4]{chen2}. In a similar way, using the setup of the proof of Corollary \ref{regularity for etale maps}, it follows that for a separated \'etale morphism $f\colon Y\rightarrow X$ of proper schemes over $k$, there is an exact functor 
    \[
    \mathsf{D}_{\mathsf{sg}}(Y)\rightarrow \mathbb{A}_f\lMod{_{\mathsf{D}_{\mathsf{sg}}(X)}},
    \]
    that is an equivalence up to retracts. 
\end{exmp}

\begin{cor} \label{equivariant singularity}
    Let $\mathcal{T}$ be a $k$-linear compactly generated triangulated category and $G$ a finite group acting on $\mathcal{T}$ such that $|G|$ is invertible in $\mathcal{T}$. Then there is an exact functor 
    \[
    (\mathcal{T}^G)^\mathsf{sg}_k \to (\mathcal{T}^\mathsf{sg}_k)^G
    \]
    which is an equivalence up to retracts. 
\end{cor}
\begin{proof}
    This follows from Theorem 
    \ref{singularity} together with the discussion before Corollary \ref{subcategories for equivariant}.
\end{proof}

\begin{exmp} \label{singularity of quotient scheme}
    Let  $X$ be a proper scheme over a commutative noetherian ring $k$, $G$ a finite group such that $|G|$ is invertible in $k$ and assume that $G$ acts freely on $X$ and that $X/G$ exists (see Subsection \ref{action on scheme} for details). Then, it follows from Example \ref{examples of singularity} and Corollary \ref{equivariant singularity} that there is an exact functor
    \[
    \mathsf{D_{sg}}(X/G) \to \mathsf{D_{sg}}(X)^G,
    \]
    which is an equivalence up to retracts. We note that similar results were observed in \cite{sun} for the singularity category of an abelian category with enough projectives.
\end{exmp}

\begin{exmp} \label{singularity of skew group dg}
    Let $A$ be a proper dg algebra over $k$ and assume an action of a finite group $G$ such that $|G|$ is invertible in $k$. Then, it follows by combining Lemma \ref{connective and proper}, Proposition \ref{derived category of dg skew}, Example \ref{examples of singularity} and Corollary \ref{equivariant singularity}, that there is an exact functor 
    \[
     \mathsf{D}_{\mathsf{sg}}(A\#G)\rightarrow \mathsf{D}_{\mathsf{sg}}(A)^G
    \]
    which is an equivalence up to retracts. 
\end{exmp}

\section{Global dimension via t-structures}  \label{Global dimension via t-structures}
It is slightly unfortunate in our setup that we can only conclude, for instance for derived categories of rings, that, under the assumptions of Corollary \ref{regularity for the derives category}, $\gd\mathsf{D}(R)<\infty$ if and only if $\gd\mathsf{D}(R\#G)<\infty$, while it is well-known that $\gd R\#G=\gd R$ (see \cite{reiten_riedtmann}). We remedy this by considering a notion of global dimension relative to a given t-structure defined in \cite{kostas}. 

\begin{defn}
    Let $\mathcal{T}$ be a triangulated category and $(\mathcal{U},\mathcal{V})$ a t-structure in $\mathcal{T}$ with heart $\mathcal{H}\coloneqq \mathcal{U}\cap \mathcal{V}[1]$. We say that the \emph{global dimension of $\mathcal{T}$ relative to $\mathcal{H}$} is equal to $n\geq 0$, denoted by $\gd_{\mathcal{H}}\mathcal{T}= n$, if 
    \[
    \mathsf{Hom}_{\mathcal{T}}(x,y[>\!n])=0, \ 
\textnormal{for all} \ x,y \in \mathcal{H}  \]
     and $n$ is the smallest integer with this property. If no such $n$ exists, then the global dimension is infinite and we write $\gd_{\mathcal{H}}\mathcal{T}=\infty$.
\end{defn}

\begin{rem}
    When $M$ is a compact silting generator of a compactly generated triangulated category $\mathcal{T}$ (i.e.\ $\mathsf{thick}(M)=\mathcal{T}^{\mathsf{c}}$ and $\mathsf{Hom}_{\mathcal{T}}(M,M[>0])=0$), then there is a t-structure $(M^{\perp_{> 0}},M^{\perp_{\leq 0}})$ with heart denoted by $\mathcal{H}_M$. If additionally $M$ belongs in $ M^{\padova}$, then $\mathcal{T}$ has finite global dimension (in the sense of Definition \ref{notions of regularity}) if and only if it has finite global dimension relative to $\mathcal{H}_M$, see \cite[Theorem 2.9]{kostas}. 
\end{rem}

We have the following result. 

\begin{prop} \label{global dimension via t structures}
    Let $\mathcal{T}$ be a triangulated category and $\mathsf{M}$ a separable exact monad on $\mathcal{T}$. Assume the existence of t-structures in $\mathcal{T}$ and $\mathsf{M}\lMod_{\mathcal{T}}\eqqcolon \mathcal{D}$, say with hearts $\mathcal{H_T}$ and $\mathcal{H_D}$ respectively, such that the functors $\mathsf{T}_{\mathsf{M}}$ and $\mathsf{U}_{\mathsf{M}}$ are t-exact. The following hold. 
    \begin{itemize}
        \item[\textnormal{(i)}] $\gd_{\mathcal{H_D}}\mathcal{D}\leq \gd_{\mathcal{H_T}}\mathcal{T}$. 
        \item[\textnormal{(ii)}] If $x$ is a summand of $\mathsf{M}(x)$ for every $x\in\mathcal{T}$, then $\gd_{\mathcal{H_T}}\mathcal{T}\leq \gd_{\mathcal{H_D}}\mathcal{D}$. 
    \end{itemize}
\end{prop}
\begin{proof}
    (i) Assume that $\gd_{\mathcal{H_T}}\mathcal{T}=n$ and let $x,y$ be objects in $\mathcal{H_D}$. Then we have  
    \[
    \mathsf{Hom}_{\mathcal{D}}(\mathsf{T}_{\mathsf{M}}\mathsf{U}_{\mathsf{M}}(x),y[>\!n])\cong \mathsf{Hom}_{\mathcal{T}}(\mathsf{U}_{\mathsf{M}}(x),\mathsf{U}_{\mathsf{M}}(y)[>\!n])\cong 0,
    \]
    since $\mathsf{U}_{\mathsf{M}}(x)$ and $\mathsf{U}_{\mathsf{M}}(y)$ both belong in $\mathcal{H_T}$, because $\mathsf{U}_{\mathsf{M}}$ is t-exact. By Lemma~\ref{summand}, we know that $x$ is a summand of $\mathsf{T}_{\mathsf{M}}\mathsf{U}_{\mathsf{M}}(x)$, so it follows that $\mathsf{Hom}_{\mathcal{D}}(x,y[>\!n])\cong 0$. Therefore $\gd_{\mathcal{H_D}}\mathcal{D}\leq n$. \\ 
    (ii) Assume that $\gd_{\mathcal{H_D}}\mathcal{D}=n$ and let $x,y$ be objects in $\mathcal{H_T}$. Then we have 
    \[
    \mathsf{Hom}_{\mathcal{T}}(x,\mathsf{U}_{\mathsf{M}}\mathsf{T}_{\mathsf{M}}(y)[>\!n])\cong \mathsf{Hom}_{\mathcal{D}}(\mathsf{T}_{\mathsf{M}}(x),\mathsf{T}_{\mathsf{M}}(y)[>\!n])\cong 0,
    \]
    since $\mathsf{T}_{\mathsf{M}}(x),\mathsf{T}_{\mathsf{M}}(y)$ belong in $\mathcal{H_D}$, as a consequence of t-exactness of $\mathsf{T}_{\mathsf{M}}$.
    Since $y$ is assumed to be a summand of $\mathsf{M}(y)=\mathsf{U}_{\mathsf{M}}\mathsf{T}_{\mathsf{M}}(y)$, it follows that $\mathsf{Hom}_{\mathcal{T}}(x,y[>\!n])=0$, and therefore $\gd_{\mathcal{H_T}}\mathcal{T}\leq n$. 
\end{proof}

\begin{exmp}
    Let $R$ be a commutative ring and $A$ a separable flat $R$-algebra. By working with the canonical t-structures in $\mathsf{D}(A)$ and $\mathsf{D}(R)$ and using Theorem \ref{etale Balmer} and Proposition \ref{global dimension via t structures}, one can see that 
    \[
    \gd A\leq \gd R. 
    \]
    When $A$ is a split extension of $R$, then Lemma \ref{split extension} implies the converse inequality. 
    
\end{exmp}

Proposition \ref{global dimension via t structures} applies, of course, to equivariant triangulated categories. 

\begin{cor} \label{global dimension for equivariant via t structure}
    Let $\mathcal{T}$ be a triangulated category and assume an action of a finite group $G$ on $\mathcal{T}$ such that $|G|$ is invertible in $\mathcal{T}$. Assume moreover the existence of t-structures on $\mathcal{T}$ and $\mathcal{T}^{G}$ with hearts $\mathcal{H}$ and $\mathcal{H}_G$ respectively, such that the functors $\mathsf{F}$ and $\mathsf{Ind}$ are t-exact. Then $\gd{_{\mathcal{H}_G}}\mathcal{T}^{G}=\gd{_{\mathcal{H}}}\mathcal{T}$.
\end{cor} 

We refer to \cite[Proposition 3.23]{sun} for sufficient conditions for a t-structured on $\T$ to induce an equivariant t-structure on $\T^G$. 

\begin{exmp}
    The derived category of a ring admits a canonical t-structure whose heart is equivalent to the module category $\Mod R$. It is direct that $\gd{_{\Mod R}}\mathsf{D}(R)=\gd R$, and by using Corollary~\ref{global dimension for equivariant via t structure} one can easily deduce that
    \[
    \gd R\#G=\gd R,
    \]
    when $G$ acts on $R$ and $|G|$ is invertible in $R$. Indeed, simply observe that the canonical t-structure of $\mathsf{D}(R)$ is $G$-invariant and that the equivariant t-structure of $\mathsf{D}(R)^G\simeq \mathsf{D}(R\#G)$ is the canonical one.  
\end{exmp} 

\begin{exmp} \label{singularity of skew group dg 2}
    For a connective dg algebra $A$ over a field $k$, its global dimension is defined to be the global dimension of $\mathsf{D}(A)$ relative to the heart induced by $A$ itself (which is a silting object), see \cite{kostas,minamoto global dimension,tomonaga}. Consider a finite group $G$ for which $|G|$ is invertible in $k$ and assume an action of $G$ on $A$. Then the adjunction 
    \[\begin{tikzcd}
	{\mathsf{D}(A\#G)} & {\mathsf{D}(A)}
	\arrow[shift right, from=1-1, to=1-2]
	\arrow[shift right, from=1-2, to=1-1]
\end{tikzcd}\]
     restricts to an adjunction between the hearts $(A\#G)^{\perp_{\neq 0}}$ and $A^{\perp_{\neq 0}}$, i.e.\ the hearts of the t-structures generated by $A\#G$ and $A$ respectively (recall from Lemma \ref{connective and proper} that $A\#G$ is connective). It follows as a consequence of Proposition \ref{derived category of dg skew} and Corollary \ref{global dimension for equivariant via t structure} that
     \[
     \gd A\#G=\gd A.
     \]
     When $A$ is locally finite (meaning that $H^n(A)\in\smod k$ for all $n$), then $A$ has finite global dimension if and only if $\mathsf{D}_{\mathsf{fd}}(A)\subseteq \mathsf{per}(A)$, see \cite[Proposition 3.4]{tomonaga} and \cite[Remark 2.20]{kostas}. As a consequence, the $k$-singularity category of $\mathsf{D}(A)$, as defined in (\ref{singularity definition}), is the quotient $\mathsf{per}(A)/\mathsf{D}_{\mathsf{fd}}(A)$. Following \cite{AGK}, we call the latter quotient the \emph{AGK singularity category} associated to $A$, and we denote it by $\mathsf{D}_{\mathsf{agk}}(A)$. Under the setup above, we also have that $A\# G$ is locally finite. We can therefore use Corollary \ref{equivariant singularity} to deduce the existence of an exact functor  
     \[
     \mathsf{D}_{\mathsf{agk}}(A\#G)\rightarrow \mathsf{D}_{\mathsf{agk}}(A)^G
     \]
     that is an equivalence up to retracts. 
\end{exmp}

\appendix
\section{Brown-Comenetz duals} \label{Brown-Comenetz duals}
Let $\mathcal{T}$ be a compactly generated triangulated category. For every compact object $t$ of $\mathcal{T}$, there is an object $t^*$, called the \emph{Brown-Comenetz dual} of $t$, such that 
\[
\mathsf{Hom}_{\mathcal{T}}(-,t^*)\cong \mathsf{Hom}_{\mathbb{Z}}(\mathsf{Hom}_{\mathcal{T}}(t,-),\mathbb{Q}/\mathbb{Z}).
\]
The map $t\mapsto t^*$ defines an exact functor $\mathrm{S}_{\mathcal{T}}\colon\mathcal{T}^{\mathsf{c}}\rightarrow \mathcal{T}$, a \emph{partial Serre functor} in the sense of \cite{OPS2}, see in particular \cite[Appendix B.]{OPS2}. The Brown-Comenetz duals play an important role in the theory of the intrinsic subcategories that we deal with. For instance, it is known that when $\mathcal{T}$ is compactly generated, then $\mathcal{T}^{\mathsf{b}}={^{\padova}}((\mathcal{T}^{\mathsf{c}})^*)$, see \cite[Lemma 3.2(ii)]{regular}. For this reason, we want to briefly investigate their behaviour under separable extensions of triangulated categories. We remark that the recent work \cite{duality and recollements} deals with a similar problem for recollements. Since every recollement of compactly generated triangulated categories gives rise to a Bousfield localization (see \cite{krause localization}) and since such a Bousfield localization is a separable extension of a smashing monad, see \cite[4.5 Remark]{balmer2}, our setup is slightly more general. 

As already explained in Remark \ref{comonad}, if we assume the setup of Theorem \ref{balmer2}, then the forgetful functor $\mathsf{U}_{\mathsf{M}}$ admits a right adjoint $\mathsf{R}_{\mathsf{M}}$ as in the diagram below. 
\[
\begin{tikzcd}
\mathcal{D} \arrow[rr, "\mathsf{U}_{\mathsf{M}}"]                     &  & \mathcal{T} \arrow[ll, "\mathsf{T}_{\mathsf{M}}"', bend right] \arrow[ll, "\mathsf{R}_{\mathsf{M}}", bend left] \\
                                                                      &  &                                                                                                                 \\
\mathcal{D}^{\mathsf{c}} \arrow[uu, "\mathrm{S}_{\mathcal{D}}", hook] &  & \mathcal{T}^{\mathsf{c}} \arrow[uu, "\mathrm{S}_{\mathcal{T}}"', hook] \arrow[ll, "\mathsf{T}_{\mathsf{M}}"]   
\end{tikzcd}
\]

\begin{lem} \label{appendix lemma}
    We have $\mathrm{S}_{\mathcal{D}}\circ \mathsf{T}_{\mathsf{M}}\simeq \mathsf{R}_{\mathsf{M}}\circ \mathrm{S}_{\mathcal{T}}$. 
\end{lem}
\begin{proof}
For every compact object $t$ of $\mathcal{T}$, we have the following isomorphisms
\begin{align*}
    \mathsf{Hom}_{\mathcal{D}}(-,\mathrm{S}_{\mathcal{D}}\mathsf{T}_{\mathsf{M}}(t))&\cong \mathsf{Hom}_{\mathbb{Z}}(\mathsf{Hom}_{\mathcal{D}}(\mathsf{T}_{\mathsf{M}}(t),-),\mathbb{Q}/\mathbb{Z}) \\ 
    & \cong \mathsf{Hom}_{\mathbb{Z}}(\mathsf{Hom}_{\mathcal{T}}(t,\mathsf{U}_{\mathsf{M}}(-)),\mathbb{Q}/\mathbb{Z})\\
    & \cong \mathsf{Hom}_{\mathcal{T}}(\mathsf{U}_{\mathsf{M}}(-),\mathrm{S}_{\mathcal{T}}(t)) \\ 
    & \cong \mathsf{Hom}_{\mathcal{T}}(-,\mathsf{R}_{\mathsf{M}}\mathrm{S}_{\mathcal{T}}(t)).
\end{align*}
The claim follows by Yoneda lemma. 
\end{proof}

Following \cite{duality and recollements}, for a compactly generated triangulated category $\mathcal{T}$, we write $\mathcal{E}_{\mathcal{T}}$ for $\mathsf{thick}(\mathrm{S}_{\mathcal{T}}(\mathcal{T}^{\mathsf{c}}))\subseteq \mathcal{T}$. 

\begin{cor}
    Under the above setup, the equality $\mathcal{E}_{\mathcal{D}}=\mathsf{thick}(\mathsf{R}_{\mathsf{M}}(\mathcal{E}_{\mathcal{T}}))$ of subcategories of $\mathcal{D}$ holds.
\end{cor}
\begin{proof}
    We observe that for any exact functor $\mathsf{F}\colon\mathcal{C}\rightarrow \mathcal{C}'$ of any triangulated categories, and any subcategory $\mathcal{X}$ of $\mathcal{C}$ we have $\mathsf{thick}(\mathsf{F}(\mathsf{thick}(\mathcal{X})))=\mathsf{thick}(\mathsf{F}(\mathcal{X}))$, which is easily obtained from the inclusion $\mathsf{F}(\mathsf{thick}(\mathcal{X}))\subseteq\mathsf{thick}(\mathsf{F}(\mathcal{X}))$.  Consequently, we have the following equalities
    \begin{align*}
        \mathcal{E}_{\mathcal{D}} & = \mathsf{thick}(\mathrm{S}_{\mathcal{D}}(\mathcal{D}^{\mathsf{c}})) \\ 
        & = \mathsf{thick}(\mathrm{S}_{\mathcal{D}}(\mathsf{thick}(\mathsf{T}_{\mathsf{M}}(\mathcal{T}^{\mathsf{c}})))) & \text{Theorem \ref{balmer2}} \\ 
        & = \mathsf{thick}(\mathrm{S}_{\mathcal{D}}\mathsf{T}_{\mathsf{M}}(\mathcal{T}^{\mathsf{c}})) \\ 
        & = \mathsf{thick}(\mathsf{R}_{\mathsf{M}}\mathrm{S}_{\mathcal{T}}(\mathcal{T}^{\mathsf{c}})) & \text{Lemma \ref{appendix lemma}} \\ 
        & = \mathsf{thick}(\mathsf{R}_{\mathsf{M}}(\mathsf{thick}(\mathrm{S}_{\mathcal{T}}(\mathcal{T}^{\mathsf{c}})))) \\ 
        & = \mathsf{thick}(\mathsf{R}_{\mathsf{M}}(\mathcal{E}_{\mathcal{T}}))
     \end{align*}
     of subcategories of $\mathcal{D}$. 
\end{proof}


\begin{thebibliography}{99}

\bibitem{amiot} 
    \textsc{C.~Amiot}, \emph{Cluster categories for algebras of global dimension 2 and quivers with potential}, Ann.\ Inst.\ Fourier (Grenoble) 59 (2009), no.\ 6, 2525–2590.


\bibitem{balmer} 
    \textsc{P.~Balmer}, \emph{Separability and triangulated categories}, Adv.\ Math.\ 226 (2011), no.\ 5, 4352–4372.

\bibitem{balmer2}
    \textsc{P.~Balmer}, \emph{The derived category of an étale extension and the separable Neeman-Thomason theorem}, J.\ Inst.\ Math.\ Jussieu 15 (2016), no.\ 3, 613–623.

\bibitem{balmer_schlichting} 
    \textsc{P.~Balmer, M.~Schlichting}, \emph{Idempotent completion of triangulated categories}, J.\ Algebra 236 (2001), no.\ 2, 819–834.

\bibitem{bhatt}  
    \textsc{B.~Bhatt}, \emph{The \'etale topology}, available at \url{https://www.math.ias.edu/~bhatt/math/etalestcksproj.pdf.}

\bibitem{bokstedt_neeman}
    \textsc{M.~Bökstedt, A.~Neeman}, \emph{Homotopy limits in triangulated categories}, Compositio Math.\ 86 (1993), no.\ 2, 209–234.
    

\bibitem{bondal_van-den-bergh} 
    \textsc{A.~Bondal, M.~van den Bergh}, \emph{Generators and representability of functors in commutative and noncommutative geometry}, Mosc. Math. J. 3 (2003), no. 1, 1–36, 258.

    \bibitem{chen2}
\textsc{X.W.~Chen},
\textit{The singularity category of an algebra with radical square zero}, Doc.\ Math.\ 16 (2011), 921--936.

\bibitem{xwchen}
\textsc{X.W.~Chen}, {\em A note on separable functors and monads with an application to equivariant derived categories}, Abh.\ Math.\ Semin.\ Univ.\ Hambg.\ 85 (2015), no.\ 1, 43--52.  


\bibitem{duality and recollements} 
    \textsc{X.~Chen, Y.~Sun, Y.~Zhang}, \emph{Triangulated categories with a compact silting object, Brown-Comenetz duality and Brown representability theorems}, arXiv:2602.14383. 

\bibitem{MerlinChrist} 
     \textsc{M.~Christ}, \emph{$\infty$-categorical group quotients via skew group algebras}, arXiv:2501.13666. 


\bibitem{eilenberg_moore} 
    \textsc{S.~Eilenberg, J.C.~Moore}, \emph{Adjoint functors and triples}, Illinois J. Math. 9 (1965), 381–398.
   
\bibitem{elagin}
    \textsc{A.~Elagin}, \emph{On equivariant triangulated categories}, arXiv:1403.7027. 


\bibitem{EGA} 
    \textsc{A.~Grothendieck}, \emph{Éléments de géométrie algébrique. IV. Étude locale des schémas et des morphismes de schémas. I.}, Inst. Hautes Études Sci. Publ. Math. No. 20 (1964), 259 pp.

    \bibitem{SGA1}
\textsc{A.~Grothendieck, M.~Raynaud}, {\em Rev\^etements \'Etales et Groupe Fondamental (SGA 1)}, S\'eminaire de G\'eom\'etrie Alg\'ebrique, vol. 1960/61, Institut des Hautes \'Etudes Scientifiques, Paris, 1963.


\bibitem{guo} 
    \textsc{L.~Guo}, \emph{Cluster tilting objects in generalized higher cluster categories}, J. Pure Appl. Algebra 215 (2011), no. 9, 2055–2071.

\bibitem{heisenberg} 
    \textsc{A.~Gyenge, C.~Koppensteiner, T.~Logvinenko}, The Heisenberg category of a category, arXiv:2105.13334. 


\bibitem{huybrechts}
    \textsc{D. Huybrechts}, {\em Fourier-Mukai Transforms in Algebraic Geometry}, Oxford University Press, USA, 2006.

\bibitem{jin} 
    \textsc{H.~Jin}, \emph{Cohen–Macaulay differential graded modules and negative Calabi–Yau configurations}, Adv.\ Math.\ 374 (2020), 107338, 59 pp.


\bibitem{AGK} 
    \textsc{H.~Jin, D.~Yang, G.~Zhou}, \emph{A localisation theorem for singularity categories of proper dg algebras}, arXiv:2302.05054.


\bibitem{keller0} 
     \textsc{B.~Keller}, \emph{Deriving DG categories}, Ann.~Sci.~Éc.~Norm.~Supér.~27, 63--102 (1994).

\bibitem{kostas} 
     \textsc{P.~Kostas}, \emph{Global dimension of dg algebras via compact silting objects}, arXiv:2604.13698. 

\bibitem{regular} 
     \textsc{P.~Kostas, C.~Psaroudakis, J.~Vit\'oria}, \emph{Intrinsic homological algebra for triangulated categories}, arXiv:2512.18417. 


\bibitem{krause localization} 
      \textsc{H.~Krause}, \emph{Localization theory for triangulated categories}, Triangulated categories, 161–235. London Math. Soc. Lecture Note Ser., 375 Cambridge University Press, Cambridge, 2010. 

\bibitem{li} 
    \textsc{L.~Li}, \emph{Homological dimensions of crossed products}, Glasg. Math. J. 59 (2017), no. 2, 401–420.

\bibitem{maclane categories} 
   \textsc{S.~MacLane}, \emph{Categories for the working mathematician}, Grad. Texts in Math., Vol. 5 Springer-Verlag, New York-Berlin, 1971. ix+262 pp.


\bibitem{milne} 
   \textsc{J.-S.~Milne}, \emph{Étale cohomology}, Princeton Math. Ser., No. 33 Princeton University Press, Princeton, NJ, 1980. xiii+323 pp.


\bibitem{minamoto global dimension} 
    \textsc{H.~Minamoto}, \emph{Resolutions and homological dimensions of DG-modules}, Israel J. Math. 245 (2021), no. 1, 409–454.

\bibitem{Mumford}
\textsc{D.~Mumford}, {\em Abelian Varieties}, With appendices by C.\ P.\ Ramanujam and Yuri Manin. Corrected reprint of the second (1974) edition. Tata Institute of Fundamental Research Studies in Mathematics, 5.\ Published for the Tata Institute of Fundamental Research, Bombay; by Hindustan Book Agency, New Delhi, 2008.\ xii+263 pp.


\bibitem{neeman} 
    \textsc{A.~Neeman}, \emph{The Grothendieck duality theorem via Bousfield's techniques and Brown representability}, J. Amer. Math. Soc. 9 (1996), no. 1, 205–236.


\bibitem{neeman7} 
   \textsc{A.~Neeman}, \emph{Triangulated categories with a single compact generator, and two Brown representability theorems}, Invent. Math. (2026), https://doi.org/10.1007/s00222-025-01401-5.


\bibitem{OPS2} 
\textsc{S.~Oppermann, C.~Psaroudakis, T.~Stai}, \emph{Partial Serre duality and cocompact objects}, Selecta Math. (N.S.) 29 (2023), no. 4, Paper No. 52, 59 pp.

\bibitem{reiten_riedtmann}
      \textsc{I.~Reiten, C.~Riedtmann}, \emph{Skew group algebras in the representation theory of Artin algebras}, J.\ Algebra 92 (1985), no.\ 1, 224–282.
 

\bibitem{sun}
    \textsc{C.~Sun}, \emph{A note on equivariantization of additive categories and triangulated categories}, J.\
Algebra 534 (2019), 483–530.

\bibitem{StacksProject} 
\textsc{The Stacks Project Authors}, {\em The Stacks Project}, 
2025, \url{https://stacks.math.columbia.edu}

\bibitem{tomonaga} 
     \textsc{R.~Tomonaga}, \emph{On silting mutations preserving global dimension}, arXiv:2510.26206. 

\end{thebibliography}
\end{document}